\documentclass[leqno]{compositio}

\usepackage[matrix,arrow,tips,curve]{xy}
\usepackage[english]{babel}
\usepackage{amsmath}
\usepackage{amssymb,amsthm}
\usepackage{enumerate}
\usepackage{psfrag}
 \psfrag{C1}{\fontsize{20}{20}$C_1$}
 \psfrag{C2}{\fontsize{20}{20}$C_2$}
 \psfrag{C3}{\fontsize{20}{20}$C_3$}

\newtheorem{thm}{Theorem}[section]
\newtheorem*{thm*}{Theorem}
\newtheorem{lemma}[thm]{Lemma}
\newtheorem{proposition}[thm]{Proposition}
\newtheorem{corollary}[thm]{Corollary}
\newtheorem{claim}[thm]{Claim}
\newtheorem*{claimstar}{Claim}

\theoremstyle{definition}
\newtheorem{definition}[thm]{Definition}
\newtheorem{parg}[thm]{}

\theoremstyle{remark}
\newtheorem{remark}[thm]{Remark}
\newtheorem{example}[thm]{Example}

\renewcommand{\qedsymbol}{$\blacksquare$}
\newcommand{\la}{\longrightarrow}

\newcommand{\ma}{\mathcal}
\newcommand{\ph}{\varphi}
\newcommand{\pr}{\mathbb{P}}
\newcommand{\ol}{\mathcal{O}}

\newcommand{\Z}{\mathbb{Z}}
\newcommand{\Q}{\mathbb{Q}}
\newcommand{\R}{\mathbb{R}}
\newcommand{\C}{\mathbb{C}}
\newcommand{\N}{\ma{N}_1}

\newcommand{\Sing}{\operatorname{Sing}}
\newcommand{\Pic}{\operatorname{Pic}}

\newcommand{\NE}{\operatorname{NE}}

\newcommand{\Exc}{\operatorname{Exc}}

\newcommand{\Chow}{\operatorname{Chow}}
\newcommand{\Lo}{\operatorname{Locus}}
\newcommand{\codim}{\operatorname{codim}}

\renewcommand{\theequation}{\thethm}

\newcounter{property}

\begin{document}
\title{Quasi elementary contractions of Fano manifolds}
\author{Cinzia Casagrande}
\email{cinzia.casagrande@unipv.it}
\address{Dipartimento di Matematica\\Universit\`a di Pisa\\largo B.\
Pontecorvo 5\\56127 Pisa - Italy}
\curraddr{Dipartimento di Matematica\\Universit\`a di Pavia\\via 
Ferrata 1\\27100 Pavia - Italy}
\shortauthors{C.\ Casagrande}
\begin{abstract}
Let $X$ be a smooth complex Fano variety. We study ``quasi
 elementary'' contractions of fiber type of $X$, 
which are a natural generalization of elementary contractions of fiber
 type. If $f\colon X\to Y$ is such a contraction, then the Picard
 numbers satisfy
$\rho_X\leq\rho_Y+\rho_F$, where 
$F$ is a general fiber of $f$. 
We show that if $\dim Y\leq 3$  and $\rho_Y\geq 4$,
 then $Y$ is smooth and Fano; if moreover $\rho_Y\geq 6$, then $X$
 is a product. This yields sharp bounds on $\rho_X$ when $\dim X=4$ and $X$
 has a quasi elementary contraction of fiber type, 
and other applications in higher
 dimensions.  
\end{abstract}
\maketitle
\section{Introduction}\label{intro}
A Fano manifold is a smooth complex projective variety $X$ with ample
anticanonical class. These varieties have a rich structure, and play an
important role in higher dimensional geometry under the viewpoint of
Mori theory, because they appear as fibers of Mori contractions of
fiber type of smooth projective varieties.

After \cite{campana,KoMiMo}
we know that 
$X$ is rationally connected and simply connected.
Moreover smooth Fano
varieties of dimension $n$ form a limited family, so they 
have only a finite number of possible topological types.
However in general very little is known
about their topological invariants.

In particular we consider here the second Betti number $b_2$ of $X$, 
which coincides with the Picard  number
$\rho_X$. 

Recall that a Del Pezzo surface $S$ has $\rho_S\leq 9$. Fano $3$-folds
have been classified by Iskovskikh, Mori, and Mukai, see
\cite{fanoEMS} and references therein. Thus
we know that a Fano
$3$-fold $X$ has $\rho_X\leq 10$. In fact, more is true: as soon as
$\rho_X\geq 6$, $X$ is a product of a Del Pezzo surface with $\pr^1$
\cite[Theorem 2]{morimukai}.

Starting from dimension $4$, we do not have a bound on $\rho_X$.
 The known examples with largest Picard number are
just products of Del Pezzo surfaces with Picard number
$9$, which gives $\rho_X=\frac{9}{2}n$. 

Optimistically one could think that Fano varieties with large Picard
number are simpler, maybe a product of lower dimensional
varieties. This would yield a linear bound (in the dimension $n$)
for $\rho_X$, in fact
one could expect precisely $\rho_X\leq\frac{9}{2}n$ (see \cite[p.\
122]{debarre}).  

This is actually what happens in the toric case: if $X$ is a smooth
toric Fano variety of dimension~$n$, then $\rho_X\leq 2n$, and
equality holds if and only if $n$ is even and $X$ is
$(S)^{\frac{n}{2}}$, $S$ the blow-up of $\pr^2$ in three non collinear
points (see
\cite{vertices}). 

Let us also recall that there is a class of Fano varieties for which a
stronger linear bound on $\rho_X$ is expected. These are Fano
varieties which do not contain curves of anticanonical degree $1$,
e.g.\ Fano varieties of index $\geq 2$. In this case it is expected
that $\rho_X\leq n$, with equality only for $(\pr^1)^n$. This is a
generalization of a conjecture by Mukai, and has been proved in
dimension $n\leq 5$ and in the toric case, see
\cite{mukai,occhettaGM,vertices}. 

\medskip

A strategy in this direction is to look for a contraction $f\colon
X\to Y$ of fiber
type on $X$, and try to bound $\rho_X$ in terms of $\rho_Y$ and
$\rho_F$, where $F$ is a general fiber. There are (at least) two
difficulties in this approach. First, $Y$ may not be Fano, so that
we do not know how to bound $\rho_Y$. Second, surely
$F$ is Fano, but in general $\rho_F$ is much smaller than
$\rho_X-\rho_Y$
(for instance any Del Pezzo surface $S$ admits a
contraction $S\to\pr^1$ with general fiber $\pr^1$). This problem
could be avoided by considering only elementary contractions of fiber
type, for which $\rho_X=1+\rho_Y$. 
However this is not very satisfactory, because we do not
necessarily expect a Fano variety with large Picard number to have an
elementary contraction of fiber type (think of a product of Del Pezzo
surfaces).

This was our motivation to introduce the notion of ``quasi
elementary''  contraction of fiber type. This is a
contraction of fiber type as above, such that if $i\colon
F\hookrightarrow X$ is a general fiber, then the image of $i_*\colon
\N(F)\to\N(X)$ contains all numerical classes of curves contracted by $f$
(see Definition \ref{quasiel}). This implies that
$\rho_X\leq\rho_Y+\rho_F$.
In particular
any elementary contraction of fiber type is quasi elementary.
If a contraction is a smooth morphism (such as a projection $Y\times
F\to Y$), then it is
quasi elementary, see Lemma \ref{smoothness}.

We study several properties of
quasi elementary contractions of smooth Fano varieties, in particular
when the target has small dimension. The following is our main result.
\begin{thm}
\label{result}
Let $X$ be a smooth complex
Fano variety, $f\colon X\to Y$
a quasi elementary contraction of fiber type, and $F$ a general fiber.
\begin{enumerate}[$(i)$]
\item
Suppose that $\dim Y=2$. 
Then $Y$ is a smooth Del Pezzo surface, $f$ is equidimensional, and
$\rho_X\leq\rho_Y+\rho_F\leq 9+\rho_F$.

If moreover $\rho_Y\geq 3$, then
$X\cong Y\times F$ and $f$ is the projection
on
the first factor. 
\item
Suppose that $\dim Y=3$. Then $\rho_X\leq
\rho_Y+\rho_F\leq
10+\rho_F$.

If moreover $\rho_Y\geq 4$, then $Y$ is smooth and Fano.

If $\rho_Y\geq 6$, then $X\cong S\times W$ and $Y\cong S\times\pr^1$,
where $S$ is a Del Pezzo surface, and $W$ a smooth Fano variety 
 with a quasi elementary contraction onto $\pr^1$.
\end{enumerate}
\end{thm}
\noindent In the case $\dim X=3$, $(i)$ has been shown in 
\cite[Proposition~8]{morimukai} under the more general assumption
that $f\colon X\to Y$ is an equidimensional contraction onto a surface.
See section 7 (p.\ \pageref{firstex}) for examples
concerning the sharpness of the statement.

The following are some applications
to Fano $4$-folds and $5$-folds.
\begin{corollary}\label{4fold}
Let $X$ be a smooth complex Fano 4-fold.

\smallskip

\noindent If $X$ has a non trivial
 quasi elementary contraction
of fiber type, then $\rho_X\leq 18$, with
equality if and only if $X\cong S_1\times S_2$, $S_i$ Del Pezzo
surfaces with $\rho_{S_i}=9$.

\smallskip

\noindent If $X$ has  an elementary contraction onto a surface $S$ and
 $\rho_X\geq 4$, then $X\cong \pr^2\times S$.

\smallskip

\noindent If $X$ has an elementary contraction onto a threefold and $\rho_X\geq
7$, then either $X\cong\pr^1\times\pr^1\times S$, or
$X\cong\mathbb{F}_1\times S$, $S$ a Del Pezzo surface.
\end{corollary}
\begin{corollary}\label{5fold}
Let $X$ be a smooth complex Fano 5-fold.

\smallskip

\noindent If $X$ has two distinct
elementary contractions of fiber type, then
$\rho_X\leq 12$.

\smallskip

\noindent If $X$ has an elementary contraction onto $Y$ with $\dim Y\leq 3$,
then $\rho_X\leq 11$.

\smallskip

\noindent Suppose that $X$ has an elementary contraction 
$f\colon X\to Y$ with $\dim Y=4$. If $X$ has another elementary
contraction $\ph$ of type (3,0), (4,0), (4,1), 
or such that $f(\Exc(\ph))=Y$, then $\rho_X\leq 12$.
\end{corollary}
\noindent 
 Finally we give an application to Fano varieties with two quasi
elementary contractions.
\begin{corollary}
\label{appl}
Let $X$ be a smooth complex Fano variety of dimension $n$. Let $f_1\colon X\to
Y_1$ and $f_2\colon X\to Y_2$ be two quasi elementary contractions
of fiber type
with $\NE(f_1)\cap\NE(f_2)=\{0\}$. Let $F_i$ be a general fiber of
$f_i$.
Then $\dim F_1+\dim F_2\leq n$, moreover:
\begin{enumerate}[$\bullet$]
\item $\dim F_1+\dim F_2=n$ implies $\rho_X\leq\rho_{F_1}+\rho_{F_2}$
\item $\dim F_1+\dim F_2=n-1$ implies $\rho_X\leq\rho_{F_1}+\rho_{F_2}+1$
\item $\dim F_1+\dim F_2=n-2$ implies $\rho_X\leq\rho_{F_1}+\rho_{F_2}+9$
\item if  $\dim F_1+\dim F_2=n-3$  and
$f_2$ is elementary, then
  $\rho_X\leq\rho_{F_1}+11$. 
\end{enumerate}
\end{corollary}
\noindent If $f_i$ is elementary one can replace $\rho_{F_i}$ by $1$
in the
statement, see  Remark \ref{elem}.
These corollaries
should be compared to the
following result by Wi{\'s}niewski, involving an arbitrary number
of elementary contractions of fiber type.
\begin{thm}[(\cite{wisn}, Theorem 2.2)]
Let $X$ be a smooth complex Fano variety of dimension~$n$. Suppose
that $X$ has $k$ distinct elementary contractions of fiber type, and
let $F_1,\dotsc,F_k$ be the general fibers. Then $\sum_i\dim F_i\leq
n$, moreover:
\begin{enumerate}[$\bullet$]
\item $\sum_i\dim F_i =n$ implies $\rho_X\leq k$
\item $\sum_i\dim F_i=n-1$ implies $\rho_X\leq k+1$.
\end{enumerate}
In particular $k\leq n$, and if $k\geq n-1$ then $\rho_X\leq n$.
\end{thm}

\smallskip

Our main tool 
is Mori theory, in particular we use many properties of
contractions of Fano varieties shown in \cite{wisn}. 

In section
\ref{prel} we recall some basic notions and properties. In
 \ref{lifting} we show that when a variety $Y$ is the target of a
contraction $f\colon X\to Y$ where $X$ is Fano, $Y$ shares with $X$ a
good behaviour with respect to Mori theory, see Lemma \ref{easy}.

In section \ref{qec} we define  quasi elementary contractions of fiber
type $f\colon X\to Y$ and give
some related properties and examples. We study the singularities of
$Y$ in Lemma \ref{target},
generalizing  results known in the elementary case and using in
particular results from \cite{ABW}. 
Then in \ref{boh} we study elementary contractions of
$Y$ by means of their liftings to $X$. This is a key ingredient in the
proof of Theorem \ref{result}. Finally in \ref{div} we
show that if $X$ is Fano, $\dim Y\geq 3$ and $Y$ contains a prime
divisor~$D$ with $\rho_D=1$, then $\rho_Y\leq 3$,
so $\rho_X\leq 3+\rho_F$ where $F$ is a general fiber of
$f$. This is a
generalization of results from \cite{toru} and \cite{bonwisncamp}.

In section \ref{quasiunsplit}  we apply
results from \cite{unsplit} to deduce that
some contractions of $X$ always induce contractions of $Y$.
This is needed in the proof of Theorem~\ref{result} $(i)$, moreover it
extends the applicability
of Theorem \ref{result}, see section \ref{ea}.

Section \ref{surf} contains the proof of Theorem \ref{result} $(i)$, 
which relies on  the
results of sections \ref{qec} and \ref{quasiunsplit}.

In section \ref{3fold} we show Theorem \ref{result} $(ii)$. 
This is based on a detailed analysis of the possible elementary
contractions of the target $Y$.
We need the
classification of smooth Fano $3$-folds by Mori and Mukai, and 
we imitate the strategy for the classification 
 of imprimitive smooth Fano $3$-folds 
(see \cite{morimukai,fanoEMS}) to get some results about the
 singular case. We also use the existence of a smoothing of 
a terminal
 Fano $3$-fold $Z$ shown in
 \cite{namikawa}, and some relations among $Z$
 and its smoothing shown in \cite{smoothings}, see Lemma \ref{conodiMori}.

Finally in section \ref{ea} we prove Corollaries \ref{4fold},
\ref{5fold}, and \ref{appl}, we give some other
application and related examples.
\begin{acknowledgements}
I wish to thank Laurent Bonavero and St\'ephane Druel for many
discussions about Fano varieties and their Picard number, and about
families of rational curves. My thanks also to Rita Pardini,
for her interest in this subject and for the conversations
 we had about it, and to
 Priska Jahnke,
from whom I learnt a lot about singular Fano threefolds.
Finally I thank the referee for useful remarks.
\end{acknowledgements}
\section{Preliminaries}
\label{prel}
We work over the field of complex numbers.

Let $X$ be a normal projective variety
of dimension $n$. We denote by $X_{reg}$ the smooth locus of~$X$.

Let $\N(X)$ be the vector space
of $1$-cycles in $X$ with real coefficients, modulo numerical
equivalence. The dimension of $\N(X)$ is equal to the Picard number
$\rho_X$ of $X$. 
Inside $\N(X)$ we denote by $\NE(X)$ the convex cone generated by
classes of effective curves, and by $\overline{\NE}(X)$ its
closure.  If $C\subset X$ is a curve, its numerical class is
$[C]\in\N(X)$. 

If $Z$ is a closed subset of $X$, call $i\colon Z\hookrightarrow X$
the inclusion, and consider the linear map
$$i_*\colon\N(Z)\la\N(X).$$
We denote by $\N(Z,X)$ the image of $i_*$ in $\N(X)$. Thus we have:
$$\dim\N(Z,X)\leq\rho_Z\quad\text{and}\quad \dim\N(Z,X)\leq\rho_X.$$

A \emph{contraction} of $X$ is a surjective morphism with connected
fibers $f\colon X\to Y$ onto a projective and normal variety $Y$. The
push-forward of $1$-cycles defined by $f$ gives a surjective linear map
$$f_*\colon\N(X)\la\N(Y),$$
so that $\rho_X-\rho_Y=\dim \ker f_*$. We also consider the convex
cone  $\NE(f)$ in $\N(X)$ generated
by classes of curves contracted by
$f$, that is 
$$\NE(f)=\NE(X)\cap\ker f_*.$$
The contraction $f$ is determined (up to isomorphism) by $\NE(f)$, see
\cite[Proposition 1.14]{debarreUT}.
We say that $f$ is \emph{of fiber type} if $\dim Y<\dim X$, otherwise
$f$ is birational. When $f$ is of fiber type, we say that $f$ is non
trivial if $\dim Y>0$.
We denote by $\Exc(f)$ the exceptional locus of
$f$, i.e.\ the locus where $f$ is not an isomorphism. We
say that $f$ is \emph{divisorial} if $\Exc(f)$ is a divisor, \emph{small}
if $\Exc(f)$ has codimension bigger that $1$.
More generally 
we say that $f$ is \emph{of type} (a,b) if $\dim\Exc(f)=a$ and $\dim
f(\Exc(f)) =b$.
Finally $f$ is \emph{elementary} if $\rho_X-\rho_Y=1$.

We will need to work with singular varieties; we refer the reader to
\cite{debarreUT,KollarMori} for the definitions and properties of
terminal and canonical singularities. We say that $X$ is
$\Q$-factorial if every Weil divisor is $\Q$-Cartier.

Suppose that $X$ has canonical singularities (in particular $K_X$ is
$\Q$-Cartier). 
We say that $f\colon X\to Y$ 
is a \emph{Mori contraction} if it is a contraction
and moreover $-K_X\cdot C>0$ for every curve $C\subset X$ contracted
by $f$. 
We recall two important properties of Mori contractions:
\stepcounter{thm}
\begin{enumerate}[(\thethm)]
\item $\dim\NE(f)=\dim\ker f_*=\rho_X-\rho_Y$,
namely $\ker f_*$ is the linear subspace generated by $\NE(f)$;
\stepcounter{thm}
\setcounter{property}{\value{thm}}
\item
for any $L\in\Pic X$ one has $L\in f^*(\Pic Y)$ if and only if $L\cdot
  C=0$ for every curve $C\subset X$ contracted by $f$.
\end{enumerate}
(See \cite[Theorem 3.7 (4)]{KollarMori} for the second statement, which
implies the first one.)

Suppose that $\NE(X)$ is closed and polyhedral. By a
\emph{face} of $\NE(X)$ we just mean a face in the geometrical
sense. For any contraction $f$ of $X$, $\NE(f)$ is a face of $\NE(X)$.
An \emph{extremal ray} is a one dimensional face, with no
assumptions on the intersection of $K_X$ with its elements.
We will use greek letters $\alpha$, $\beta$, etc.\ to denote faces of
$\NE(X)$. If $\alpha$ is an extremal ray and $D$ a $\Q$-Cartier
divisor on $X$, we will say that $D\cdot\alpha >0$, $D\cdot\alpha=0$,
or $D\cdot\alpha<0$,  if respectively $D\cdot
v>0$, $D\cdot v=0$, or $D\cdot v<0$ for a non zero element $v\in\alpha$.
We denote by $\Lo(\alpha)\subseteq X$ the union of all curves in $X$
whose numerical class is in $\alpha$.

Let $X$ be a projective variety with  canonical
singularities and $K_X$ Cartier. We say that
$X$ is \emph{Fano} if $-K_X$ is ample. If so, 
the cone $\NE(X)$ is closed and polyhedral, and any
contraction of $X$ is a Mori contraction. Moreover for any face
$\alpha$ of $\NE(X)$ there exists a contraction $f$ of $X$ such that
$\alpha=\NE(f)$. This follows from the Contraction Theorem, see
\cite[Theorem 3.7]{KollarMori}.
\begin{remark}
\label{wellknown}
Let $X$ and $Y$ be factorial projective varieties, and $\sigma\colon
X\to Y$ the blow-up of a smooth subvariety $A\subset Y_{reg}$. Suppose
that $X$ is Fano and let $C\subset Y$ be an irreducible curve such
that $C\not\subset A$ and $C\cap A\neq\emptyset$. Let $\widetilde{C}$
be the proper transform of $C$ in $X$. Then:
$$-K_X\cdot\widetilde{C}<-K_Y\cdot C,$$
in particular $-K_Y\cdot C\geq 2$.

In fact if $E=\Exc(\sigma)$ we have $\widetilde{C}\not\subseteq E$ and 
 $\widetilde{C}\cap E\neq\emptyset$, so $\widetilde{C}\cdot
 E>0$. Moreover $K_X=\sigma^*(K_Y)+aE$ where $a=\codim A-1$, so
$$ -K_X\cdot \widetilde{C}= -K_Y\cdot C -aE\cdot \widetilde{C}<-K_Y\cdot C.$$
\end{remark}
\begin{remark}
\label{divisorial}
Let $X$ be a normal and $\Q$-factorial
projective variety and $f\colon X\to Y$ an
elementary divisorial contraction. Then $\Exc(f)$ is an irreducible
divisor
and $\Exc(f)\cdot\NE(f)<0$. 

In fact let $E$ be an irreducible component of $\Exc(f)$, 
then $E\cdot\NE(f)<0$
 (see for instance \cite[Lemma 3.39]{KollarMori} applied to $B=-E$).
Now if $C\subset X$ is an irreducible curve contracted by $f$, we have 
$E\cdot C<0$, thus $C\subseteq E$. Hence $E=\Exc(f)$.
\end{remark}
\begin{parg}\label{lifting}
{\bf Targets of contractions of smooth Fano varieties.}
Let $X$ be a smooth Fano variety and
$f\colon X\to Y$ a contraction. 
Consider
the push-forward
$f_*\colon\N(X)\to\N(Y)$.
We observe that
$$f_*(\NE(X))=\NE(Y),$$
namely \emph{$\NE(Y)$ is the linear projection of $\NE(X)$ from the
  face $\NE(f)$}. This simple remark implies many properties of
$\NE(Y)$. For instance it is closed and polyhedral, since $\NE(X)$
is. Moreover, faces of $\NE(Y)$ are in bijection (via $f_*$) with
faces of $\NE(X)$ containing $\NE(f)$. In fact this description
is the same
 as the one involving the ``star of a cone'' in toric geometry, see \cite[p.\
  52]{fulton}. 

Let's consider a face $\alpha$ of $\NE(Y)$, and let $\widehat{\alpha}$ be the
unique face of $\NE(X)$ containing $\NE(f)$ and such that
$f_*(\widehat{\alpha})=\alpha$. Then
$\dim\widehat{\alpha}=\dim\alpha+\dim\NE(f)$. Since 
$\NE(f)$ is a face of $\widehat{\alpha}$, we can choose another face
$\widetilde{\alpha}$ of $\widehat{\alpha}$ with the properties:
$$\dim\widetilde{\alpha}=\dim\alpha \quad\text{ and }\quad
\widetilde{\alpha}\cap\NE(f)=\{0\}.$$
Observe that the choice of $\widetilde{\alpha}$ will not be unique in
general, and that it can very well be
$\widetilde{\alpha}+\NE(f)\subsetneq \widehat{\alpha}$. 

When $\alpha$ is an extremal ray, $\widetilde{\alpha}$ is an extremal
ray of $\NE(X)$, and $f_*(\widetilde{\alpha})=\alpha$. 
There is a rational curve $C\subset X$ such that
$[C]\in\widetilde{\alpha}$, hence $f(C)$ is a rational curve in $Y$
with numerical class in $\alpha$.

Since $X$ is Fano, there exist contractions $\ph\colon X\to W$ and
$h\colon X\to Z$ such that $\widetilde{\alpha}=\NE(\ph)$ and
$\widehat{\alpha}=\NE(h)$: 
$$\xymatrix{
X\ar[d]_f\ar[dr]^{h}\ar[r]^{\ph}& W \\
Y & Z}$$
Now by rigidity (see for instance \cite[Lemma 1.15]{debarreUT}) there exist
contractions $\psi\colon Y\to Z$ and $g\colon W\to Z$ that make the
following diagram commute:
$$
\xymatrix{
X\ar[d]_f\ar[dr]^{h}\ar[r]^{\ph}& W\ar[d]^g \\
Y\ar[r]_{\psi} & Z}
$$
It is not difficult to check that $\NE(\psi)=\alpha$ and that
$\dim\ker\psi_*=\dim\alpha$.  We will say that  
$\ph\colon X\to W$ is a \emph{lifting} of
$\psi$.
Summing up, 
 we
have proved the following.
\begin{lemma}\label{easy}
Let $X$ be a smooth Fano variety and $f\colon X\to Y$ a
contraction. Then $\NE(Y)$ is a closed polyhedral cone, and every
extremal ray contains the class of some rational curve. 

Moreover for every face $\alpha$ of $\NE(Y)$ there exists a
contraction $\psi\colon Y\to Z$ such that $\NE(\psi)=\alpha$ and
$\rho_Y-\rho_Z=\dim\alpha$. 
\end{lemma}
\end{parg}
\section{Quasi elementary contractions}
\label{qec}
Let $X$ be a smooth variety and $f\colon X \to Y$ a Mori
contraction of fiber type.  Recall that: 
$$\rho_X-\rho_Y=\dim\NE(f)=\dim \ker f_*.$$
Let $F$ be a general fiber and consider 
$\N(F,X)\subseteq\N(X)$. 
We have 
$$\N(F,X)\subseteq\ker f_*,$$
hence $\dim \N(F,X)\leq \rho_X-\rho_Y$.
\begin{definition}
\label{quasiel}
We say that $f$ is \emph{quasi elementary} if 
$$
\N(F,X)=\ker f_*,$$
equivalently if $\dim\N(F,X)=\rho_X-\rho_Y$.
Since $\dim\N(F,X)\leq\rho_F$, if $f$ is quasi elementary we get:
$$
\rho_X\leq \rho_Y+\rho_F.
$$
\end{definition}
\begin{example}
\label{compo}
Let $f\colon X\to Y$ be a Mori contraction of fiber type.
\begin{enumerate}[$\bullet$]
\item If $f$ is elementary, then it is also 
quasi elementary.

This is because $1\leq \dim\N(F,X)\leq\dim\ker f_*=1$.
\item
Suppose that
$\dim X-\dim Y=1$. Then $f$ is quasi elementary if and only if it is
elementary. 
\item
Suppose that $X$ is Fano and $f$ is quasi elementary. If 
$\psi\colon Y\to
Z$ is an elementary contraction of fiber type, then the composition
$\psi\circ f\colon X\to Z$ is quasi elementary. 

In fact since $\psi$ is elementary we have $\dim\ker(\psi\circ f)_*=\dim\ker
f_*+1$. 
 Let $G$ be a general fiber of $\psi\circ f$, then $G\supset F$
 hence $\ker(\psi\circ f)_*\supseteq
\N(G,X)\supseteq\N(F,X)=\ker
f_*$. Choosing a curve $C\subset G$ not
contracted by $f$ shows that $\N(G,X)\supsetneq\N(F,X)$, so
$\N(G,X)=\ker(\psi\circ f)_*$ and $\psi\circ f$ is quasi elementary.
\end{enumerate}
\end{example}
Let's show that the notion of quasi elementary is related to smoothness.
\begin{lemma}
\label{smoothness}
Let $X$ be a smooth variety and $f\colon X\to Y$ a Mori contraction
of fiber type. 
Let $Y_0\subseteq Y_{reg}$ be an open subset over which $f$ is smooth,
and set 
 $X_0:=f^{-1}(Y_0)$. 
If $\codim(X\smallsetminus X_0)\geq 2$ and $Y$ is $\Q$-factorial,
then $f$ is quasi elementary. 
\end{lemma}
We remark that the converse to Lemma \ref{smoothness} does not
hold (take for
instance a smooth Fano 3-fold $X$ with an elementary
contraction $X\to\pr^1$ which is not a smooth morphism). 
Moreover the hypothesis of $\Q$-factoriality on $Y$ is 
necessary, see the contraction $f\colon V\to Y$ in
example \ref{delpezzo}.
\begin{proof}
Set $f_0:=f_{|X_0}\colon
X_0\to Y_0$. 
Following the notation of \cite[\S 0-1]{KMM}, consider
$$\NE(X_0/Y_0)\subset\N(X_0/Y_0).$$
Since $f_0$ is a Mori contraction, $\NE(X_0/Y_0)$ is closed and
polyhedral by the relative version of the Cone Theorem, see 
\cite[Theorem 4-2-1]{KMM}. 

On the other hand $\N(X/Y)=\ker f_*\subseteq\N(X)$, and the inclusion 
$X_0\hookrightarrow X$ induces a
natural injective homomorphism
$$\N(X_0/Y_0)\hookrightarrow \ker f_*.$$
With a slight abuse of notation, we will consider $\N(X_0/Y_0)$ as a
subspace of $\N(X)$.

For every $y\in Y_0$ we have $i_*(\NE(F_y))\subseteq\NE(X_0/Y_0)$. 
Consider an  
 extremal ray $\alpha$ of $\NE(X_0/Y_0)$. Then \cite[Proposition
1.3]{wisndef} says that $\Lo(\alpha)$ dominates $Y_0$ via $f$, hence
$\alpha\subseteq i_*(\NE(F_y))$ for every $y\in Y_0$. Repeating this
for every extremal ray of $\NE(X_0/Y_0)$, we get
$$i_*(\NE(F_y))=\NE(X_0/Y_0)$$
for every $y\in Y_0$, in particular 
\stepcounter{thm}
\begin{equation}
\label{darcy}
\N(F_y,X)=\N(X_0/Y_0)\subseteq\ker f_*\subseteq\N(X).\end{equation}

The homomorphism dual to the inclusion $\N(X_0/Y_0)\hookrightarrow
\ker f_*$
is
$$r\colon \frac{(\Pic X)\otimes{\R}}{f^*(\Pic Y)\otimes{\R}}
\la\frac{(\Pic
  X_0)\otimes{\R}}{f_0^*(\Pic Y_0)\otimes{\R}}$$ 
induced by the restriction $\Pic X\to\Pic X_0$
(see \cite[\S 0-1 and Lemma 3-2-5 (2)]{KMM}). 
Clearly $r$ is surjective.

Let $L\in\Pic X$ be 
  such that $L_{|X_0}=f_0^*(M_0)$ for some $M_0\in\Pic Y_0$. Since $Y$
  is $\Q$-factorial, there exists $M\in\Pic Y$ such that
  $M_{|Y_0}=M_0^{\otimes l}$, $l\in\Z_{\geq 1}$. Then $L^{\otimes
  l}\otimes f^*(M)^{\otimes(-1)}$ 
is trivial on $X_0$, and by our hypothesis trivial on
  $X$. Thus $r$ is an isomorphism.
Dually, this is gives $\N(F_y,X)=\N(X_0/Y_0)=\ker f_*$, so $f$ is
quasi elementary.
\end{proof}
\begin{remark}
Observe that \eqref{darcy} shows that $\N(F_y,X)$ does not depend on
$y\in Y_0$, so 
that the condition $\N(F_y,X)=\ker f_*$ can be
checked for an arbitrary $y\in Y_0$. 
\end{remark}
\begin{remark}
If $f\colon X\to Y$ is quasi elementary, then $\N(F_0,X)=\ker f_*$ for
\emph{every} fiber $F_0$ of $f$ (with the reduced structure). 

In fact we have $\N(F_0,X)\subseteq\ker f_*$. Moreover if
$f_0\colon X_0\to Y_0$ is as in the proof of Lemma \ref{smoothness},
$\NE(X_0/Y_0)$ has dimension $\rho_X-\rho_Y$. For any fixed
extremal ray $\alpha$ of $\NE(X_0/Y_0)$,  \cite[Proposition
1.3]{wisndef} says that $\Lo(\alpha)$ dominates $Y_0$. Then taking a
family of curves whose class is in $\alpha$ and their degenerations,
we see that $\alpha\subset\N(F_0,X)$. This implies 
that $\N(F_0,X)=\ker
f_*$.
\end{remark}
\begin{remark} 
Suppose that $X$ is smooth and Fano, $f\colon X\to Y$  a
contraction of fiber type, and $F$ a general fiber.

In general the push-forward $i_*\colon\N(F)\to\N(X)$ does not need
to be injective: for instance there are smooth Fano threefolds that
have an elementary contraction onto $\pr^1$, with fibers Del Pezzo
surfaces with $\rho>1$. 
This is related to the monodromy of the fibration $f$. 

Consider an
open subset $Y_0$ as in Lemma \ref{smoothness} and let $y\in Y_0$.
The dimension of $\N(F_y,X)$ is equal to the dimension of the image of
the restriction
$$ (\Pic X)\otimes{\Q}=H^2(X,\Q)\la H^2(F_y,\Q)=(\Pic F_y)
\otimes{\Q}.$$
In turn this is equal to the dimension of the linear
subspace of $H^2(F_y,\Q)$ which is invariant for the monodromy
action of $\pi_1(Y_0,y)$ (see for instance \cite[Chapter~15]{voisin}). 
Hence $\dim\N(F_y,X)=\rho_F$ if and only if the monodromy
action is trivial.
\end{remark}
\begin{example}
Let $X$ be a smooth Fano variety, $f\colon X\to Y$ a non trivial
contraction of
fiber type with $Y$ smooth,  and $F$ a general fiber.
Suppose that $f$ is smooth outside a
finite number of points of $Y$, and that there are no fibers
of codimension~$1$.
Then $Y$ is Fano, $f$ is quasi elementary, and $\rho_X=\rho_Y+\rho_F$.

In fact $Y$ is Fano by \cite[Theorem 3]{miyaoka93}, in particular it
is simply connected, and $\dim Y\geq 2$.
Then
$Y\smallsetminus\{y_1,\dotsc,y_m\}$
 stays simply connected, so the monodromy action on 
$H^2(F,\Q)$ is trivial, and $\dim\N(F,X)=\rho_F$. On the other hand 
$f$ is quasi elementary 
by Lemma \ref{smoothness}, so $\rho_X=\dim\N(F,X)+\rho_Y=
\rho_F+\rho_Y$.
\end{example}
The following two lemmas give some basic properties of quasi
elementary contractions.
\begin{lemma}
\label{nonequilocus}
Let $X$ be a smooth variety and $f\colon X\to Y$ a quasi elementary
Mori contraction of fiber type.
\begin{enumerate}[$(i)$]
\item If $D$ is a prime divisor in $X$ such that $f(D)\subsetneq Y$,
  then $f(D)$ is a Cartier divisor, and $D=f^*(f(D))$.
\item
The locus where $f$ is not equidimensional has codimension at
least $3$ in $Y$.
\end{enumerate}
\end{lemma}
\begin{proof}
 Since $f$ is quasi elementary and
$D$ is disjoint from the general fiber, we have $D\cdot C=0$ for every
curve $C\subset X$ contracted by $f$. By property (\ref{prel}.\theproperty) of
Mori contractions,
there exists an effective
Cartier divisor $D'$ on $Y$ such that $D=f^*(D')$. 
Moreover ${D}'$ is supported on
  $f(D)$, so there exists $m\in\Z_{>0}$ such that ${D}'=mf(D)$.

Set $U:=f^{-1}(Y_{reg})$ and 
observe that $X\smallsetminus U$ has codimension at least two in $X$.
In $Y_{reg}$ the intersection 
$f(D)\cap Y_{reg}$ is a (non trivial) Cartier divisor
 and $D_{|U}=(f_{|U})^*(mf(D))_{|Y_{reg}}=
m(f_{|U})^*(f(D)\cap Y_{reg})$. Since $X$ is smooth,
there exists an effective divisor $D''$ in $X$ such that $D''_{|U}= 
(f_{|U})^*(f(D)\cap Y_{reg})$. Then $D_{|U}=m D''_{|U}$, so $D=m D''$
and $m=1$ because $D$ is prime.

Let's show $(ii)$. Let $K\subset Y$ be 
an irreducible closed subset
such that the general
fiber of $f$ over $K$ has dimension at least $\dim X -\dim Y+1$. Then
$$\dim X-1\geq \dim f^{-1}(K)\geq \dim X -\dim Y+1+\dim K,$$
so $\dim K\leq\dim Y-2$. If by contradiction $\dim K=\dim Y-2$, then 
$\dim f^{-1}(K)\geq \dim X-1$.
Consider a prime divisor $D$ contained
in $f^{-1}(K)$: we have $f(D)\subseteq K$, which contradicts $(i)$.
\end{proof}
We generalize to quasi elementary contractions some known properties
of elementary contractions of fiber type.
In particular $(ii)$ is shown  
in \cite{ABW} in the elementary case.
\begin{lemma}
\label{target}
Let $f\colon X\to Y$ be a Mori contraction of fiber type, with $X$ smooth.
\begin{enumerate}[$(i)$]
\item
If  $f$ is quasi elementary, then $Y$ is factorial and
  has canonical singularities.
\item 
If $Y$ is a surface and $f$ is equidimensional, then $Y$ is smooth.
\item
If $Y$ is a surface and $f$ is quasi elementary, then $f$ is
equidimensional and $Y$ is smooth.
\item If $\dim Y=3$ and $f$ is quasi elementary, then $Y$ has isolated
  singularities. 
\end{enumerate}
\end{lemma}
\begin{proof}
$(i)$ 
 Let $D$ be a prime Weil divisor in $Y$ and let $D'$ be a prime divisor
  in $X$ such that $f(D')=D$. Then Lemma \ref{nonequilocus} $(i)$
 yields that $D$ is Cartier.

Thus $Y$ is factorial, in particular $K_Y$ is Cartier. 
It is known that $Y$ has rational
singularities, see \cite[Corollary 7.4]{kollarhigher}. Then $Y$ has
also canonical singularities, see \cite[Corollary 5.24]{KollarMori}.

\medskip

$(ii)$
When $f$ is elementary, this is \cite[Proposition
  1.4.1]{ABW}. In general, observe that $Y$ is a normal surface with 
rational singularities, in particular it is $\Q$-factorial and has
isolated singularities. 

We want to show that in fact $Y$ has quotient
singularities, using results from \cite{watanabe} and
\cite{flennerzaid}.
More precisely, let $y_0\in Y$ be a singular point. In
\cite[Definition 1.4]{watanabe} 
and \cite[Definition~1.25]{flennerzaid} one can find the definition
of the plurigenera $\delta_m(Y,y_0)$ of $Y$ in $y_0$, for $m\in\Z_{> 0}$. 
Since $f$ is equidimensional,
then it is ``non degenerate'' in the sense of \cite[Definition
  1.14]{flennerzaid}. Then \cite[Corollary~1.27]{flennerzaid} gives
$\delta_m(Y,y_0)=0$ for every $m\in\Z_{>0}$. This is equivalent to
saying that $y_0$ is a quotient singularity by \cite[Theorem
  3.9]{watanabe}. 

Hence $Y$ has quotient singularities, and we can apply
\cite[Proposition~1.4]{ABW} to deduce that $Y$ is actually
smooth. Observe that in \cite{ABW} the contraction $f$ is assumed to
be elementary, but the proof works word for word in the case of an
equidimensional Mori contraction. Let us also remark that by the
definition of quotient singularities, we can cover $Y$ by open subsets
in the complex topology which are quotients of an open subset of
$\C^2$ by a finite group. Hence in the proof of \cite[Proposition
  1.4]{ABW} 
one has actually to work in the analytic category, however
everything works in the same way.

\medskip

$(iii)$  This follows immediately from Lemma \ref{nonequilocus} and $(ii)$.

\medskip

$(iv)$ By Lemma \ref{nonequilocus}, $f$ can have at most isolated
fibers of dimension $n-2$. Let $S\subset Y$ be a general hyperplane
section and $D:=f^{-1}(S)$. Then $S$ is a normal surface (see \cite[Lemma
5.30]{KollarMori}) and $f_{|D}\colon D\to S$ is
equidimensional. Moreover 
$D$ is general in a base point free
linear system, so it is smooth. 
Since $D$ is disjoint from the general fiber of $f$ and $f$ is quasi
elementary, we have $D\cdot C=0$ for every curve $C\subset X$
contracted by $f$. In particular if $C\subset D$ we get
$$-K_D\cdot C=-K_X\cdot C>0,$$
so $f_{|D}$ is an equidimensional Mori contraction. Then $S$ is smooth
by $(ii)$, and this yields $\dim\Sing Y=0$.
\end{proof}
Thus the target of a quasi elementary contraction has reasonable
singularities.
The following simple remark will be very useful.
\begin{remark}\label{fano}
Let $X$ be a smooth Fano variety and $f\colon X\to Y$ a quasi
elementary contraction, so that $Y$ is factorial and has canonical
singularities by Lemma \ref{target}.

Then $Y$ is Fano if and only if $-K_Y\cdot\alpha>0$ for every extremal
ray $\alpha$ of $\NE(Y)$.

Equivalently, $Y$ is Fano if and only if every elementary contraction
$\psi\colon Y\to Z$ is a Mori contraction.

This is a straightforward consequence of Lemma \ref{easy}.
\end{remark}
\begin{parg}
\label{boh}
{\bf Elementary contractions of the target.} 
Let $X$ be a smooth projective variety and $f\colon X\to Y$ a quasi
elementary Mori contraction of fiber type. Recall that $Y$ is
factorial with canonical singularities by Lemma \ref{target}.

Let $\psi\colon Y\to Z$ be a contraction. We say that
\emph{$\psi$ has a lifting} if there exist a Mori contraction 
 $\ph\colon X\to Y$ such that $\NE(\ph)\cap\NE(f)=\{0\}$,
 $\rho_X-\rho_W=\rho_Y-\rho_Z$, and a commutative diagram:
\stepcounter{thm}
\begin{equation}
\label{pizza}
\xymatrix{
X\ar[d]_f
\ar[r]^{\ph}& W\ar[d]^g \\
Y\ar[r]_{\psi} & Z}\end{equation}
We will also say that $\ph$ is a \emph{lifting} of $\psi$.

Notice that $\psi$ is elementary if and only if $\ph$ is elementary. 
We are interested in comparing properties of $\ph$ and $\psi$.

When $X$ is Fano, any contraction
$\psi\colon Y\to Z$ has a lifting, 
as explained in \ref{lifting}.
Conversely, in Theorems \ref{face} and \ref{face2}  we will give
conditions on a Mori
contraction $\ph\colon X\to W$ to be a lifting of some $\psi$.
\begin{thm}
\label{secondo}
Let $X$ be a smooth projective variety and $f\colon X\to Y$ a quasi
elementary Mori contraction of fiber type.

Let $\psi\colon Y\to Z$ be a
birational elementary contraction 
with fibers of dimension at most $1$. Assume that $\psi$ has a lifting
$\ph\colon X\to W$ as in \eqref{pizza}.
Then the following holds:
\begin{enumerate}[$(i)$]
\item $W$ is smooth and $\ph$ is the blow-up of a smooth subvariety of
  codimension $2$;
\item $g$ is a quasi elementary Mori contraction
and $Z$ is factorial with canonical
  singularities; 
\item $\psi$ is divisorial,
$\Exc(\ph)=f^*(\Exc(\psi))$, and $-K_Y\cdot \NE(\psi)\geq 0$.
\item If $\dim Y=3$, then
 $\psi$ is the blow-up of a smooth curve $C\subset Z_{reg}$.
\item If $\dim Y=2$, then:
\begin{enumerate}[$(v.a)$]
\item
$Y$ and $Z$ are smooth, $\psi$ is the blow-up of $z_0\in Z$,
and $g$ has smooth fiber $F_0$ over $z_0$;
\item
 $\ph$ is the blow-up of $F_0$,
  $\Exc(\ph)\cong\pr^1\times F_0$,
and
$f_{|\Exc(\ph)}$, $\ph_{|\Exc(\ph)}$ are the two projections;
\item if $X$ is Fano, then $W$ is Fano.
\end{enumerate}
\end{enumerate}
\end{thm}
\begin{proof}
 Let $F$ be a non trivial fiber of $\ph$. Then $f$ is finite on $F$ and
$f(F)$ is contained in a non trivial fiber of $\psi$, hence 
$f(F)\subseteq\Exc(\psi)$ and $\dim F=1$. 
Since $\Exc(\ph)$ is covered by non
trivial fibers of $\ph$, we get
$f(\Exc(\ph))\subseteq\Exc(\psi)$. In particular $\ph$ is birational
with fibers of dimension at most $1$, so
\cite[Theorem 1.2]{wisn}
yields that $W$ is smooth and $\ph$ is the blow-up of a smooth subvariety of
  codimension $2$.

Let $E\subset X$ be the exceptional divisor of $\ph$,
then $f(E)$ is contained in $\Exc(\psi)$ and it is a divisor by Lemma
\ref{nonequilocus} $(i)$. Hence $\psi$ is
elementary and divisorial, and $\Exc(\psi)$ is irreducible by Remark
\ref{divisorial},  that is 
 $\Exc(\psi)=f(E)$. Then again Lemma~\ref{nonequilocus} $(i)$ gives
 $E=f^*(\Exc(\psi))$.

We have
$$K_X= \ph^*(K_W)+ E.$$
Let $C$ be an irreducible curve contracted by $g$ and $\widetilde{C}$
an irreducible curve in $X$ such that $\ph(\widetilde{C})=C$. Then
$\ph_*(\widetilde{C})=m C$ with $m\in\Z_{>0}$. Moreover $f(\widetilde{C})$ is
a point, so $\widetilde{C}\cdot E=0$. Then:
$$m(-K_W\cdot
C)=-K_W\cdot\ph_*(\widetilde{C})=(-K_X-E)\cdot \widetilde{C}=
 -K_X\cdot \widetilde{C}>0,$$
so $g$ is a Mori contraction.

We have $\dim\ker
f_*=\rho_X-\rho_Y=\rho_W-\rho_Z=\dim\ker g_*$. 
Moreover $\NE(\ph)$ is an extremal ray of $\NE(X)$
not contained in $\NE(f)$, hence $\ker\ph_*\cap\ker f_*=\{0\}$. 
Then $\dim \ph_*(\ker
f_*)=\dim\ker g_*$, on the other hand
 $\ph_*(\ker f_*)\subseteq\ker g_*$, so equality holds. 

Observe that $\ph$ is an isomorphism over the general fibers of $f$
and $g$. Let $F_1$ be a general fiber of $f$, so that $\N(F_1,X)=\ker
f_*$ because
$f$ is quasi elementary.
 Then $\ph(F_1)$ is a general fiber of
$g$, and $\N(\ph(F_1),W)=\ph_*(\N(F_1,X))=\ph_*(\ker f_*)
=\ker g_*$. Thus $g$ is quasi
elementary. 

Then $Z$ is factorial with canonical singularities by Lemma
\ref{target} $(i)$.
Let's show that $-K_Y\cdot\NE(\psi)\geq 0$. Write
$K_Y=\psi^*(K_Z)+r\Exc(\psi)$ with  $r\in\Z$, and let $C_1\subset Y$
be a curve contracted by $\psi$. Then $\Exc(\psi)\cdot C_1<0$ by
Remark
\ref{divisorial}, and $-K_Y\cdot C_1=r(-\Exc(\psi)\cdot C_1)$. Hence we
have to show that $r\geq 0$. Let $h\colon Y'\to Y$ be a
resolution of singularities of $Y$, and consider the composition
$\psi\circ h\colon Y'\to Z$. Call $E'$ the proper transform of
$\Exc(\psi)$ in $Y'$. Then $r$ is the coefficient of $E'$ in
$K_{Y'}-(\psi\circ h)^*(K_Z)$, thus $r\geq 0$ because $Z$ has
canonical singularities.

\medskip

Assume now that $\dim Y=3$ and
set $S:=\Exc(\psi)$. 

Let's first notice that $S_{reg}=S\cap Y_{reg}$. This is because if
$y_0\in S_{reg}$, then $y_0$ must be smooth for $Y$
too, because $S$ is a Cartier divisor in $Y$. On the other hand let $y_1\in
Y_{reg}$ and let $h\in\mathcal{O}_{Y,y_1}$ be a local equation for $S$
in $y_1$. Since $E=f^*(S)$, $f^*(h)$ is a local equation for $E$
near the fiber over $y_1$. Take $x_1$ in this fiber. Since $E$ is
smooth, the differential $d_{x_1}(f^*(h))$ is non zero, 
and since $d_{x_1}(f^*(h))=d_{y_1}h\circ d_{x_1}f$, 
the
differential $d_{y_1}h$ must be non zero too. Thus $S$ is smooth at
$y_1$. 

Recall that $Y$ has rational singularities, in particular it is
Cohen-Macaulay (see \cite[Corollary~7.4]{kollarhigher}
and \cite[Theorem 5.10]{KollarMori}). Since $S$ is a Cartier divisor
in $Y$, it is Cohen-Macaulay too. On the other hand,
Lemma \ref{target} $(iv)$ implies that $Y$ and $S$ have
isolated singularities. Then $S$ is normal by Serre's criterion.

By Lemma~\ref{nonequilocus} $f$ can
have at most isolated fibers of dimension $n-2$.
Let's show that $f$ is equidimensional over $S$. 
If $F_0\subset E$ is an irreducible component of a fiber of
$f$ with $\dim F_0=n-2$, then $\ph(F_0)\subseteq \ph(E)$ and
$\dim\ph(F_0)=n-2=\dim\ph(E)$, hence $\ph(F_0)=\ph(E)$. This gives
$$\psi(S)=
\psi(f(E))=g(\ph(E))=g(\ph(F_0))=\psi(f(F_0))=pt,$$
a contradiction.

Now $E$ is smooth, $S$ is normal, and  $f_{|E}\colon 
E\to S$ has connected fibers.
As in the proof of Lemma~\ref{target} $(iv)$ we
see that
$f_{|E}$ is a Mori
contraction, and by Lemma \ref{target} $(ii)$ the surface $S$ is
smooth, so that 
 $S\cap\Sing(Y)=\emptyset$.

The general fiber $l$ of  $\psi_{|S}\colon S\to \psi(S)$ is a rational
curve, and it is smooth because $S$ is, hence $-K_S\cdot l=2$.

Let $\tilde{l}$ be a fiber of $\ph$ such that $f(\tilde{l})=l$. Then
$$-1=\tilde{l}\cdot E=\tilde{l}\cdot f^*(S)=f_*(\tilde{l})\cdot S,$$
which gives $l\cdot S=-1$ and $-K_Y\cdot l=-K_S\cdot l+S\cdot l=1$. 

This shows also that $\psi_{|S}\colon S\to\psi(S)$ is a $\pr^1$-bundle. Therefore
$\psi$ is
the blow-up of a smooth curve contained in the smooth locus of
$Z$, and $Z$ has the same singularities as $Y$.

\medskip

Finally, let's suppose that $\dim Y=2$.
Then $Y$ and $Z$ are smooth by Lemma~\ref{target} $(iii)$, and $f$ and
$g$ are equidimensional. Hence $\psi$ is the blow-up of a point
$z_0\in Z$.
We have 
$$E=f^{-1}(\Exc(\psi))=\ph^{-1}(g^{-1}(z_0)),$$
so the center of $\ph$ is $\ph(E)=g^{-1}(z_0)=F_0$. 

Since $E=f^*(\Exc(\psi))$, $E$ is the (schematic)
fiber  of  $\psi\circ f$ over $z_0$, and it is reduced. Thus
$g\circ\ph$ has reduced fiber over $z_0$, and the same must hold for
$g$.
Now $g$ is an equidimensional morphism between smooth varieties, whose
fiber $F_0$ over $z_0$ is reduced and smooth. This implies that $g$ is
smooth over $z_0$, and hence the normal bundle of $F_0$ in $W$ is
trivial. Since $E$ is the exceptional divisor of the blow-up of $F_0$,
we deduce that
$E\cong
\pr^1\times F_0$, and $f_{|E}$ and
$\ph_{|E}$ are the two projections.

\medskip

Suppose that $X$ is Fano and $W$ is not.
Then \cite[Proposition 3.4]{wisn} 
says that there exists an extremal ray
$\alpha$ of $\NE(X)$, different from $\NE(\ph)$, 
such that ${\alpha}\cdot
E<0$. This implies that $\alpha$ is not contained in $\NE(f)$,
because $C\cdot E=0$ for  every curve $C$ contracted by $f$.
However 
$E\cong 
\pr^1\times F_0$, so every curve in $E$ is numerically equivalent
to a linear combination with coefficients in $\Q_{\,\geq 0}$ of a curve
in $\NE(f)$ and a curve in $\NE(\ph)$. Thus no curve in $E$ can have
numerical class in an extremal ray $\alpha$ not contained in
$\NE(f)\cup\NE(\ph)$, and we have a contradiction.
\end{proof}
\end{parg}
\begin{parg}\label{div}
{\bf Divisors ${D}$ with small ${\dim\N(D,X)}$.} 
Let $X$ be a smooth Fano variety of dimension $n$, 
and $D$ a prime divisor in $X$. 
If $D$ is simple, e.g.\ if $\rho_D$ is very small, then
one can hope to deduce
informations on $X$ itself.
For instance
in \cite{bonwisncamp} the authors classify the 
possible pairs $(X,D)$ when
$D\cong \pr^{n-1}$ and
$\mathcal{N}_{D/X}\cong\mathcal{O}_{\pr^{n-1}}(-1)$. This is
equivalent to asking that $X$ is obtained by blowing-up a
smooth variety in a point. 

This classification has been generalized in \cite{toru} to the case 
$D\cong \pr^{n-1}$ and
$\mathcal{N}_{D/X}\cong\mathcal{O}_{\pr^{n-1}}(-a)$ with $a\in\Z_{\geq
  1}$. 

With the same techniques \cite[Proposition 5]{toru} shows that if $X$
contains a 
prime divisor $D$ with $\rho_D=1$ and $\dim X\geq 3$, then $\rho_X\leq 3$.
The proof of
\cite[Proposition 5]{toru} can be generalized in order to obtain the
following. 
\begin{proposition}
\label{divisor}
Let $X$ be a smooth Fano variety, and let $f\colon X\to Y$ be
either a quasi elementary
 contraction of fiber type or an isomorphism.

Suppose that $\dim Y\geq 3$ and that 
there exists a prime divisor $D\subset Y$ such that
$\dim\N(D,Y)=1$. Then $\rho_Y\leq 3$.
\end{proposition}
\noindent In the case where $f$ is an isomorphism, this says that if
$\dim X\geq 3$ and
$X$  contains a prime divisor~$D$ with $\dim\N(D,X)=1$, then $\rho_X\leq 3$.
In particular if $\dim X=n\geq 3$ and
$X$ has an elementary contraction of type $(n-1,0)$,
then $\rho_X\leq 3$.
\begin{proof}
Recall that $Y$ is factorial by Lemma \ref{target} $(i)$. Moreover
$\NE(Y)$ is closed and polyhedral, and every extremal ray can be
contracted, by Lemma \ref{easy}.

Let's assume that $\rho_Y>1$.
There exists at least one extremal ray $\alpha_1$ of $\NE(Y)$ such that
$\alpha_1\cdot D>0$, let $\psi_1\colon Y\to Z_1$ be its contraction. Then
$\dim Z_1>0$ and 
$D$ must intersect every non trivial fiber of~$\psi_1$.

If $\alpha_1\subset\N(D,Y)$, then every curve in $D$ has numerical class
contained in $\alpha_1$, hence $\psi_1(D)$ is a point. Then $\psi_1$
can not be of fiber type (otherwise $Z_1$ is a point), so $\psi_1$ is
birational and $D\subseteq\Exc(\psi_1)$. This implies
$D\cdot\alpha_1<0$ by Remark \ref{divisorial}, a contradiction.  

Hence $\alpha_1\not\subset\N(D,Y)$. Consider a non
trivial fiber $F$ of $\psi_1$. Then $F\cap D\neq\emptyset$, on the other
hand $\dim(F\cap D)=0$ otherwise we would get a curve in $D$ with
numerical class in $\alpha_1$. Since $Y$ is factorial, we get
$$0=\dim(F\cap D)\geq\dim F-1,$$ 
so $F$ has dimension one. 

If $\psi_1$ is of fiber type, then
$$(\psi_1)_{*|\N(D,Y)}\colon\N(D,Y)\la \N(Z_1)$$
is surjective, so $\rho_{Z_1}= 1$ and $\rho_Y= 2$.

Assume that $\psi_1$ is birational
and consider a lifting of $\psi_1$ as in \ref{lifting}:
$$\xymatrix{
X
\ar[dr]^h
\ar[r]^{\ph_1}\ar[d]_f & {W_1}\ar[d]^g \\
Y\ar[r]_{\psi_1} & {Z_1}
}$$
Since $X$ is Fano, $\ph_1$ is a Mori contraction, and Theorem
\ref{secondo} applies. Thus $g$ is a quasi elementary Mori
contraction,  
 $W_1$ is smooth, $\ph_1$ is a blow-up,
$\psi_1$ is divisorial,
$Z_1$ is factorial, and  $\Exc(\ph_1)=f^{-1}(\Exc(\psi_1))$.

Since $\psi_1$ is finite on $D$,
$D_2:=\psi_1(D)$ is a Cartier
divisor in $Z_1$. Moreover
$$(\psi_1)_{*|\N(D,Y)}\colon\N(D,Y) \la\N(D_2,Z_1)$$
is surjective, hence $\dim\N(D_2,Z_1)=1$. 
Recall also that $D$ intersects every non trivial fiber
of $\psi_1$, hence $D_2\supset\psi_1(\Exc(\psi_1))$. 

Again by Lemma \ref{easy}, $\NE(Z_1)$ is closed and polyhedral, and
every extremal ray can be contracted.

Thus consider 
an extremal ray $\alpha_2$ of $\NE(Z_1)$ such that $\alpha_2\cdot
D_2>0$, and let $\psi_2\colon Z_1\to Z_2$ be the associated contraction.

As before, we see that if $\psi_2$ is of fiber type then $\rho_{Z_1}\leq 2$ and
 $\rho_Y\leq 3$.

If $\psi_2$ is birational, then $\alpha_2\not\subset\N(D_2,Z_1)$ and 
again every non
trivial fiber of $\psi_2$ has dimension one. 
We show that this case leads to a contradiction.

Let $\widehat{\alpha}_2$ be the unique extremal face of $\NE(X)$ such that
$\widehat{\alpha}_2$ contains $\NE(h)$ and
$h_*(\widehat{\alpha}_2)=\alpha_2$ (see \ref{lifting}). 
Notice that $\NE(\ph_1)$ 
is an extremal ray of  $\NE(h)$.

Let's make an easy remark on convex
polyhedral cones. If $\sigma\subset\R^d$ is a 
convex
polyhedral cone, $\gamma$ is a proper face of $\sigma$, and $\rho$ is
an extremal ray of $\gamma$, then there exists an extremal ray $\rho_2$
of $\sigma$, not contained in $\gamma$, and such that $\rho+\rho_2$ is a
face of $\sigma$. The reader who 
is not familiar with convex geometry may easily
prove this by induction on $\dim\sigma$.

Applied to our situation, this says that there is an extremal ray
 $\widetilde{\beta}$ of $\widehat{\alpha}_2$ such that
$\widetilde{\beta}\not\subset\NE(h)$ and $\widetilde{\beta}+\NE(\ph_1)$
 is a face of $\widehat{\alpha}_2$. 

Then $\beta=(\ph_1)_*(\widetilde{\beta})$ is an extremal ray of
$\NE(W_1)$, whose 
contraction $\ph_2\colon W_1\to W_2$ yields a commutative
diagram:
$$\xymatrix{
X\ar[dr]^h\ar[r]^{\ph_1}\ar[d]_f & {W_1}\ar[d]^g\ar[r]^{\ph_2}& {W_2}\ar[d] \\
Y\ar[r]_{\psi_1} & {Z_1}\ar[r]_{\psi_2} &{Z_2}
}$$
Observe that $\NE(\psi_2\circ h)=\widehat{\alpha}_2$ and
$\NE(\ph_2\circ\ph_1)=\widetilde{\beta}+\NE(\ph_1)\subseteq
\NE(\psi_2\circ h)$, so the morphism $W_2\to Z_2$ exists by rigidity,
see \ref{lifting}.
Since $\psi_2$ is birational and has fibers of
dimension at most one, it is easy to see that the same holds for
$\ph_2$.

Suppose that there exists a non trivial fiber $C$ of $\ph_2$ contained in
$\ph_1(\Exc(\ph_1))$.
Let $\widetilde{C}$ be an irreducible curve in $\Exc(\ph_1)$ such that
$\ph_1(\widetilde{C})=C$. Then $f(\widetilde{C})\subset
f(\Exc(\ph_1))= \Exc(\psi_1)$ and $h(\widetilde{C})\subseteq
\psi_1(\Exc(\psi_1))\subset D_2$. 

On the other hand, since $\ph_2(C)$ is a point,
$\psi_2(h(\widetilde{C}))$ must also be a point, namely 
the numerical class of $h(\widetilde{C})$ must be in
$\alpha_2=\NE(\psi_2)$. 
But this contradicts $\alpha_2\not\subset\N(D_2,Z_1)$. 

Therefore if $C$ is a non trivial fiber of $\ph_2$, we have
$C\not\subseteq\ph_1(\Exc(\ph_1))$,
and this implies that 
$-K_{W_1}\cdot C>0$. 

Thus
 $\ph_2$ is a Mori contraction, and Theorem  \ref{secondo} yields that
$W_2$ is smooth, $\ph_2$ is the blow-up of a smooth subvariety of
codimension $2$ in $W_2$, and $\Exc(\ph_2)=g^{-1}(\Exc(\psi_2))$. 
In particular $-K_{W_1}\cdot C=1$.

Now using Remark \ref{wellknown} we see that any non trivial fiber $C$ 
of $\ph_2$ can not intersect 
$\ph_1(\Exc(\ph_1))$, hence:
$$
\Exc(\ph_2)\cap \ph_1(\Exc(\ph_1))=\emptyset.
$$
 Recall that $\Exc(\ph_1)=f^{-1}(\Exc(\psi_1))$, which gives
$\ph_1(\Exc(\ph_1))=g^{-1}(\psi_1(\Exc(\psi_1)))$. Therefore
$$\emptyset=g^{-1}\bigl(\Exc(\psi_2)\bigr)
\cap g^{-1}\bigl(\psi_1(\Exc(\psi_1))\bigr)=g^{-1}\bigl(\Exc(\psi_2)
\cap \psi_1(\Exc(\psi_1))\bigr),$$
namely
 $\Exc(\psi_2)\cap \psi_1(\Exc(\psi_1))=\emptyset$.

Let's notice that $\dim\psi_1(\Exc(\psi_1))=\dim Y-2\geq 1$, so there
exists a curve $C_2\subseteq \psi_1(\Exc(\psi_1))$, and $C_2\cdot
\Exc(\psi_2)=0$. 

On the other hand $\psi_1(\Exc(\psi_1))\subset D_2$ and
$\dim\N(D_2,Z_1)=1$. This implies that $C_2'\cdot \Exc(\psi_2)=0$ for
every curve $C_2'\subset D_2$. 
But this is impossible, because $D_2$ and $\Exc(\psi_2)$ are distinct
prime divisors with non empty intersection.
\end{proof}
\begin{remark}
\label{dettaglio}
More precisely we have shown that in the hypotheses of 
Proposition~\ref{divisor}, either $\rho_Y=1$, or one of the following occurs:
\begin{enumerate}[$\bullet$]
\item $\rho_Y=2$ and $Y$ has an elementary contraction of fiber type with
$1$-dimensional fibers;
\item $\rho_Y=3$ and $Y$ has a divisorial elementary contraction such that
every fiber has dimension at most~$1$.
\end{enumerate}
\end{remark}
\end{parg}
\section{Existence of contractions after \cite{unsplit}}
\label{quasiunsplit}
Let $X$ be any normal projective variety. 
Consider an irreducible closed subset $V$ of $\Chow(X)$ such that:
\begin{enumerate}[$\bullet$]
\item
for $v\in V$ general, the corresponding cycle $C_v\subset X$ is an
irreducible, reduced, and connected rational curve;
\item
every point of $X$ is contained in $C_v$ for some $v\in V$.
\end{enumerate}
We call $V$ a \emph{covering family of rational
curves} in $X$. Such family induces an
equivalence relation on $X$ (as a set), called $V$-equivalence, as
follows. Two points $x,y\in X$ are $V$-equivalent if
there exist $v_1,\dotsc,v_m\in V$ such that $C_{v_1}\cup\cdots\cup
C_{v_m}$ is connected and contains $x$ and $y$. This notion was 
originally
introduced by Campana \cite{campana81}
and is by now well known, see \cite[\S IV.4]{kollar},
\cite[\S 5]{debarreUT}, \cite{campa}. We refer
specifically to \cite{unsplit} for the set up and for precise
references.

In particular it is known that there exist: an open subset $X_0$ which is
closed under $V$-equivalence, a normal quasi-projective variety $Y_0$,
and a proper, equidimensional morphism $q\colon X_0\to Y_0$, such that every
fiber of $q$ is a $V$-equivalence class. In general, there are no
morphisms defined on the whole $X$ which extend $q$.

If $f\colon X\to Y$ is a Mori contraction of fiber type, then one can
find a family $V$ as above such that $q=f_{|X_0}$ (see the proof of
Theorem \ref{face}). Using the properties of this family, we can apply
the results of \cite{unsplit} to deduce the following.
\begin{thm}
\label{face}
Let $X$ be a smooth variety of dimension $n$ and consider two Mori
contractions
$$\xymatrix{X\ar[d]_f\ar[r]^{\ph} & W\\
Y & }$$
where $f$ is elementary of fiber type and 
$\NE(f)\cap \NE({\ph})=\{0\}$. 
Let $k_f$ and $k_{\ph}$ be the
dimensions of the general fibers of $f$ and ${\ph}_{|\Exc(\ph)}$ 
respectively.

Assume that we are in one of the
following situations:
\begin{enumerate}[$(i)$]
\item ${\ph}$ is quasi elementary
of fiber type, 
and $k_f+k_{\ph}\geq n-3$;
\item ${\ph}$ is
  elementary and divisorial, and $\dim Y\leq 3$;
\item ${\ph}$ is
  elementary and divisorial, $f(\Exc(\ph))=Y$,
and $k_{\ph}\geq\dim Y -3$.
\end{enumerate}
Then  there exists a
commutative diagram:
\stepcounter{thm}
\begin{equation}\label{diag}
\xymatrix{
X\ar[r]^{\ph}\ar[d]_f\ar[dr]^{h} & W\ar[d]^g \\
Y \ar[r]& Z
}\end{equation}
where $g\colon W\to Z$ is an elementary Mori contraction and
$\dim Z\leq 3$.
\end{thm}
\noindent Let us point out that we do not know whether the hypotheses
on $k_{f}+k_{\ph}$, $\dim Y$, and $k_{\ph}$ respectively are really
necessary for the statement to hold.
\begin{proof}
In the first part of the proof we will just assume that $f$ is quasi
elementary and that $\NE(f)\cap \NE({\ph})=\{0\}$.

We first construct a suitable covering family $V$
of rational curves in $X$,
such that $f$ is the quotient for
$V$-equivalence over an open subset of $X$.

Let $F$ be a general fiber of $f$, then $F$ is a smooth Fano variety of
dimension $k_f$. In particular $F$ is rationally connected, and there
exists a smooth rational curve $C_0\subset F$ which is very free,
namely with ample normal bundle in $F$:
$$\mathcal{N}_{C_0/F}\cong
\bigoplus_i\mathcal{O}_{\pr^1}(a_i),\quad
a_i\in\Z_{>0}\text{ for every }i$$
(see \cite[\S 4.3 and 5.6]{debarreUT}).
Since $\mathcal{N}_{F/X}$ is trivial, we have
$$
\mathcal{N}_{C_0/X}\cong\mathcal{O}_{\pr^1}^{\oplus(n-k_f)}
\oplus\bigoplus_{i}\mathcal{O}_{\pr^1}(a_i),$$
hence $C_0$ is a free curve in $X$. This means that the deformations
of $C_0$ cover the whole $X$:
by \cite[Proposition
4.8]{debarreUT} there exists a covering family $V$ of rational curves
in $X$ such that $C_0=C_{v_0}$ for some $v_0\in V$.

Clearly all curves parametrized by $V$ are numerically equivalent in
$X$, and their  numerical class is contained
in $\NE(f)$, because $f(C_0)$ is a point.
 Since $\NE(f)$ is a face of $\NE(X)$, we  deduce that 
every irreducible component of 
every curve
parametrized by $V$ is contained in some fiber of $f$. 
This implies that
every
$V$-equivalence class is contained in some fiber of $f$. 

Consider the quotient for $V$-equivalence $q\colon X_0\to Y_0$, let
$F_0$ be a general fiber, and let $C_v\subset F_0$ general.
Then
$C_v\cong\pr^1$ and we have: 
$$\mathcal{N}_{C_v/X}\cong \mathcal{O}_{\pr^1}^{\oplus(n-\dim F_0)}
\oplus\mathcal{N}_{C_v/F_0}.$$
On the other hand,
by \cite[II.3.9.2]{kollar} the number of trivial summands in 
$\mathcal{N}_{C_v/X}$ is at most the number of trivial summands in
$\mathcal{N}_{C_0/X}$, so that $n-\dim F_0\leq n-k_f$ and $\dim F_0\geq
k_f$.

This
implies that $\dim F_0=k_f$ and hence
all fibers of $f$ of dimension
$k_f$ are $V$-equivalence classes.

\medskip

We observe that $V$ induces in a natural way a covering family of
rational curves on $W$. 
First of all, the hypothesis $\NE(f)\cap\NE({\ph})=\{0\}$ implies that ${\ph}$
does not contract any irreducible component of any curve of $V$.

Consider now the incidence
diagram associated to $V$:
$$\xymatrix{\mathcal{C}\ar[d]\ar[r] &X \\
V &}$$
We proceed as in \cite[proof of Lemma 2]{unsplit}.
 Let $\widetilde{\mathcal{C}}$ be the normalization of
$\mathcal{C}$ and $\widetilde{\mathcal{C}}\to 
\widetilde{V}$ be the Stein factorization of
the composite map $\widetilde{\mathcal{C}}\to\mathcal{C}\to V$. Then
$\widetilde{V}$  is
normal, the general fiber of
$\widetilde{\mathcal{C}}\to\widetilde{V}$  is $\pr^1$, and the
composite map
$$\xymatrix{{\widetilde{\mathcal{C}}}
\ar[r]&{\mathcal{C}}\ar[r]&X\ar[r]^{\ph} & W }$$
yields a family of $1$-cycles in $W$, thus a morphism
$\widetilde{V}\to\Chow(W)$.
We call $V'$ the image of this morphism.
By the construction, 
for every irreducible component $C$
of a curve parametrized by $V$, the image $\ph(C)$ is a component
of some
curve parametrized by $V'$, and conversely.
In particular, if $x,y\in X$ are $V$-equivalent, then ${\ph}(x),{\ph}(y)\in W$
are $V'$-equivalent. 
\begin{claimstar}
\label{dimension}
Let $T\subseteq W$ be a
general $V'$-equivalence class. 
Then $\codim T\leq n-k_f$, and
$\codim T\leq n-(k_f+k_{\ph})$ if $f(\Exc(\ph))=Y$.
\end{claimstar}
\noindent Observe that the second case always holds if $\ph$ is of fiber type.
\begin{proof}[Proof of the Claim]
The inverse image ${\ph}^{-1}(T)$ is closed for
$V$-equivalence. Let $F$ be a general $V$-equivalence class contained
in ${\ph}^{-1}(T)$, so that $F$ is a fiber of $f$.
Since ${\ph}$ is finite on $F$ and ${\ph}(F)\subseteq T$, we have 
$\dim T\geq \dim F=k_f$, while $\dim W\leq n$, which gives the first statement.

If $\ph$ is of fiber type, then $\dim W=n-k_{\ph}$, so the
same argument gives $\codim T\leq n-(k_f+k_{\ph})$.

Let's assume that $\ph$ is birational and 
that $\Exc(\ph)$ dominates $Y$ via
  $f$. 
Then $F\cap\Exc(\ph)\neq\emptyset$, 
so $T\cap
  \ph(\Exc(\ph))\neq \emptyset$, and since $T$ is general, every
  $V'$-equivalence class intersects $\ph(\Exc(\ph))$.  
This means that ${\ph}^{-1}(T)$
  contains a general fiber $N$ of $\ph_{|\Exc(\ph)}$.
 Let $N_0$ be an irreducible component of~$N$ with $\dim
  N_0=k_{\ph}$.

There exists a non empty open subset $U_0$ of $f(N_0)$ such that every fiber of
$f$ over $U_0$ is an irreducible $V$-equivalence class. Then
$f^{-1}(U_0)$ is irreducible of dimension
$k_f+k_{\ph}$, and it must be contained in $\ph^{-1}(T)$
because  ${\ph}^{-1}(T)$ is closed for
$V$-equivalence. Moreover 
$\ph_{|f^{-1}(U_0)}$ is birational, so $\dim T\geq k_f+k_{\ph}$.
\end{proof}
Consider now the linear subspaces $H_V$ of $\N(X)$ and $H_{V'}$ of $\N(W)$
generated by the
numerical classes of all irreducible components of all curves in
$V$ and $V'$ respectively. Then we have ${\ph}_*(H_V)=H_{V'}$, 
moreover  $\N(F,X)\subseteq H_{V}$ and
$\N(T,W)\subseteq H_{V'}$ by \cite[Proposition IV.3.13.3]{kollar}
(see also \cite[Remark 1]{unsplit}).

Since
 $f$ is quasi elementary, we get
 $\ker f_*=\N(F,X)\subseteq H_V\subseteq\ker f_*$, 
hence $\N(F,X)= H_V=\ker f_*$.
  Then we have:
\stepcounter{thm}
\begin{gather}
\notag
\N(T,W)\subseteq H_{V'} 
={\ph}_*(H_V)={\ph}_*(\N(F,X))=\N({\ph}(F),W)\subseteq  
\N(T,W),\\
\label{pippo}
\text{hence }\qquad\N(T,W)=H_{V'}={\ph}_*(\N(F,X))= {\ph}_*(\ker f_*).
\end{gather}

We observe that ${\ph}$ is either elementary and
divisorial, or quasi elementary of fiber type. Then in any case $W$ is
$\Q$-factorial and has canonical singularities,
by Lemma \ref{target} and \cite[Proposition~7.44]{debarreUT}.

\medskip

Let's assume now that $f$ is an
elementary contraction. Then
$\dim H_V=\dim H_{V'}=1$, and this means
that $V$ and $V'$ are quasi-unsplit in the sense of
\cite{unsplit}. Moreover, 
using the Claim we see that in any case 
the general $V'$-equivalence class has 
codimension at most $3$.
Hence \cite[Theorem 2]{unsplit} yields the existence of 
an elementary Mori contraction $g\colon W\to Z$ such
that every fiber of $g$ is a $V'$-equivalence class. Then $\dim Z\leq
3$, and
by rigidity (see for instance \cite[Lemma~1.15]{debarreUT}) there
exists a morphism $Y\to Z$ as in the 
statement.
\end{proof}
We observe that with a slight modification of the argument, we can
prove a different version of the previous Theorem. Namely we can
allow $f$ to be quasi elementary instead of elementary, if we impose a
stronger condition on $k_f$ and $k_{\ph}$.
\begin{thm}
\label{face2}
Let $X$ be a smooth variety of dimension $n$ and consider two Mori
contractions $f\colon X\to Y$, $\ph\colon X\to W$
such that $\NE(f)\cap \NE({\ph})=\{0\}$ and $f$ is quasi elementary of
fiber type. 
 Let $k_f$ and $k_{\ph}$ be the
dimensions of the general fibers of $f$ and ${\ph}_{|\Exc(\ph)}$ respectively.

Assume that we are in one of the
following situations:
\begin{enumerate}[$(i)$]
\item ${\ph}$ is quasi elementary and
  $k_f+k_{\ph}\geq n-2$;
\item ${\ph}$ is
  elementary and divisorial, and $\dim Y\leq 2$;
\item ${\ph}$ is
  elementary and divisorial, $f(\Exc(\ph))=Y$,
and $k_{\ph}\geq\dim Y-2$.
\end{enumerate}
Then  there exists a
commutative diagram as \eqref{diag}
where $h\colon X\to Z$ is a contraction, $\dim Z\leq 2$, and
$\rho_X-\rho_Z\leq(\rho_X-\rho_Y)+(\rho_X-\rho_W)$ 
(equality holds except possibly in $(i)$). 
\end{thm}
\begin{proof}
We perform the same construction as in the previous proof, so that
$W$ is a normal and $\Q$-factorial projective variety, and $V'$
is a covering family of rational curves in $W$.
Moreover by the Claim and the assumptions, we see  that now
the general $V'$-equivalence class has 
codimension at most $2$ in $W$. On the other hand,
$V'$ is not quasi-unsplit if
$f$ is not elementary.
However
\eqref{pippo} implies the following property.
\begin{quote}
\emph{Let $D$ be a Weil divisor in $W$ whose support is
 disjoint from the general $V'$-equivalence class. Then $D\cdot C=0$
 for every irreducible component of every curve in $V'$.}
\end{quote}
We claim that we can apply \cite[Proposition 1]{unsplit} 
to the family $V'$ on $W$, even if $V'$ is not quasi-unsplit.
Indeed the quasi-unsplit assumption
is used in the proof of this Proposition 
uniquely to deduce the property above. One can think
that the property above generalizes the quasi-unsplit property, in the
same way as quasi elementary contractions generalize elementary
contractions of fiber type.

Now as in the proof of \cite[Theorem 1]{unsplit}, we get 
 the existence of a normal
projective variety $Z$
and a surjective morphism $g\colon W\to Z$ such
that every fiber of $g$ is a $V'$-equivalence class (even if $g$ is
not necessarily a Mori contraction).

As before this implies the existence of the diagram \eqref{diag} by
rigidity, and $h:=g\circ\ph$ is a contraction.
Moreover it
is clear that $\ker g_*=H_{V'}$ and
$$ \ker h_*=({\ph}_*)^{-1}(H_{V'})=({\ph}_*)^{-1}({\ph}_*(\ker f_*))=\ker
f_*+\ker {\ph}_*.$$
Observe that the hypothesis $\NE(f)\cap\NE(\ph)=\{0\}$ does not imply
$\ker f_*\cap\ker \ph_*=\{0\}$, unless we know that one of $f$ or
$\ph$ is elementary, as in $(ii)$ or $(iii)$. Thus $\dim\ker h_*\leq
\dim \ker f_*+\dim\ker \ph_*$, and equality holds except possibly in
$(i)$. 
\end{proof}
\begin{remark}
\label{dim}
In both Theorems \ref{face} and \ref{face2}, we have in fact proved
that $\dim Z\leq\dim Y$ in case $(ii)$, and $\dim Z\leq
n-(k_f+k_{\ph})=\dim Y-k_{\ph}$ in cases $(i)$ and $(iii)$.
Thus the contraction $Y\to Z$ is of fiber type in case $(i)$,
and elementary of fiber type in case $(iii)$.
\end{remark}
\section{Fano manifolds with a quasi elementary contraction onto a surface}
\label{surf}
In this section we show Theorem \ref{result} $(i)$.
So let's consider a smooth Fano variety $X$ of dimension $n\geq 3$,
 and 
 a quasi elementary contraction $f\colon X\to Y$ onto a surface. 
We know by Lemma \ref{target} $(iii)$ that $f$ is equidimensional and
$Y$ is smooth. 
Moreover $Y$ is rational, for instance
because $X$ is rationally connected.
\begin{remark}
\label{nuoto}
Let $\psi\colon Y\to Z$ be a birational elementary contraction.
Then $Z$ is smooth, $\psi$ is the blow-up of a point $z_0\in
Z$, and
there exists a diagram
$$
\xymatrix{ X \ar[d]_f \ar[r]^{\ph} & W\ar[d]^g \\
             Y \ar[r]^{\psi} & {Z}}
$$
where $W$ is a smooth Fano variety,
$g$ a quasi elementary contraction with smooth fiber $F_0$ over $z_0$,
 $\ph$ is the blow-up of $F_0$,
  $\Exc(\ph)\cong\pr^1\times F_0$,
and
$f_{|\Exc(\ph)}$, $\ph_{|\Exc(\ph)}$ are the two projections.

In fact $\psi$ has a lifting as explained in 
\ref{lifting}, which gives the diagram above.
Since $X$ is Fano, $\ph$ is a Mori contraction. Then
Theorem \ref{secondo} $(v)$ gives the statement.

Notice moreover that the lifting $\ph$ is uniquely determined by
$\psi$, because $\Exc(\ph)=f^{-1}(\Exc(\psi))\cong\pr^1\times F_0$ and
$\NE(\ph)$ is determined by the curves $\pr^1\times\{pt\}$ in $X$. 
As described in \ref{lifting}, every choice of an extremal ray 
$\widetilde{\alpha}$ in $\NE(\psi\circ f)$ such that
$\widetilde{\alpha}\not\subseteq\NE(f)$, gives rise to a lifting
of $\psi$. We deduce that such an extremal ray 
$\widetilde{\alpha}=\NE(\ph)$ is unique, and that $\NE(\psi\circ f)=
\widetilde{\alpha}+\NE(f)$.
\end{remark}
\begin{proof}[Proof of Theorem \ref{result} $(i)$]
If $\rho_Y=1$ then $Y\cong\pr^2$, so we can 
assume that $\rho_Y>1$.
In order to show that $Y$ is Del Pezzo, 
it is
enough to show that any elementary contraction of $Y$ is a Mori contraction,
see Remark \ref{fano}.

Let's fix such a contraction $\psi\colon Y\to Z$. 
If $\psi$ is birational, then Remark \ref{nuoto} says that 
$Z$ is smooth and $\psi$ is a
blow-up, hence a Mori contraction.

Let's consider the case where $\psi$
is not birational, 
so that $\dim Z\leq 1$. Since we have assumed $\rho_Y>1$,
we have $\dim Z=1$, and $Y$ is a smooth rational surface with
$\rho_Y=2$. Thus $Y$ is a Hirzebruch surface and
$Z\cong\pr^{1}$.
 If we
consider the other elementary contraction of $Y$, we find two
possibilities: either it is again a contraction over $\pr^1$ and
$Y\cong\pr^1\times\pr^1$, or it is birational. In this case, by
Remark~\ref{nuoto} it must
be a smooth blow-up, so $Y\cong \mathbb{F}_1$.

Hence $Y$ is a Del Pezzo surface, and we have the first part of the
statement.

\medskip

Let's assume that $\rho_Y\geq 3$ and 
prove the second part of the statement. Since $Y$ is Del Pezzo,
there is a morphism $Y\to\pr^2$ which is a
blow-up of $\rho_Y-1$ distinct points $p_i$. 
Thus we can apply Remark \ref{nuoto} to each blow-up 
and get a smooth Fano variety $X_0$ with a quasi
elementary contraction $f_0\colon X_0\to\pr^2$, 
such that $X$ is obtained from $X_0$
by blowing-up the fibers of $f_0$ over $p_i$. 

Therefore it is enough to prove the statement in the case
$\rho_Y=3$,
where $Y$ is the blow-up of $\pr^{2}$ in two points $p_1$ and $p_2$. 
Call $C_1$ and $C_2$ the two corresponding exceptional curves, and
$C_3$ the proper transform in $Y$ of the line through $p_1$ and $p_2$
in $\pr^2$. 

\smallskip

\begin{center}
{}{}\scalebox{0.30}{\includegraphics{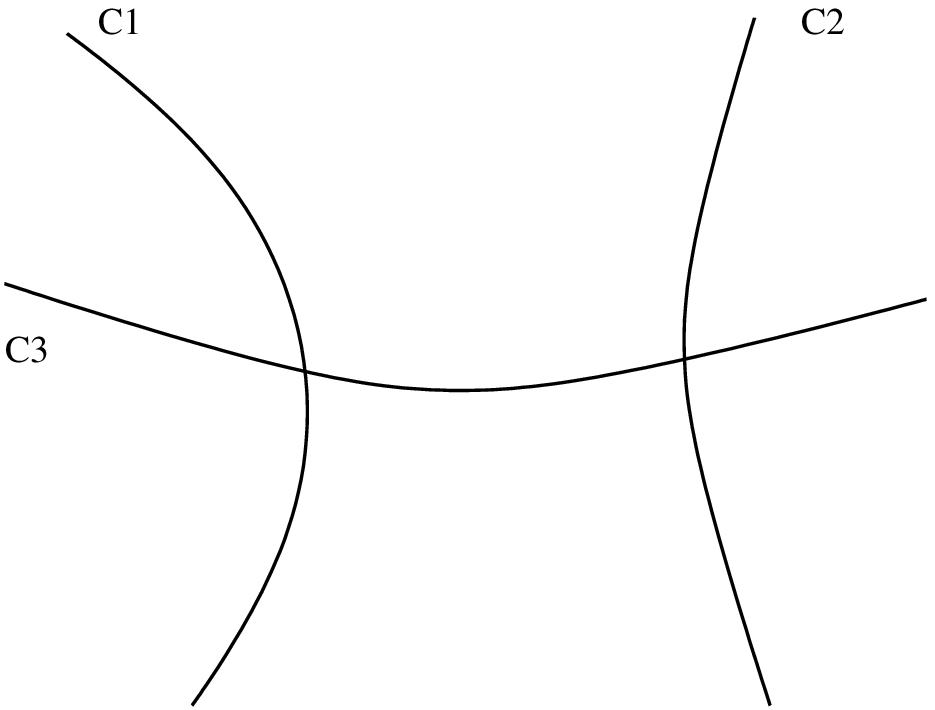}}
\end{center}

\noindent These are all the  
$(-1)$-curves in $Y$. Call $\alpha_i$ the extremal ray of $\NE(Y)$
containing the numerical class of $C_i$; then
$\NE(Y)=\alpha_1+\alpha_2+\alpha_3$ 
is simplicial.

\medskip

Let's make a preliminary remark on the cone $\NE(X)$. 
The surface $Y$ has three elementary contractions, which are blow-ups
with exceptional curves $C_1$, $C_2$, and $C_3$.
Recall from Remark \ref{nuoto} that the lifting of each of these
blow-ups is unique. This means
that for every $i=1,2,3$ there 
is a unique extremal ray
$\widetilde{\alpha}_i$ of $\NE(X)$ such that 
$\widetilde{\alpha}_i+\NE(f)$ is a face of $\NE(X)$ and
$f_*(\widetilde{\alpha}_i)={\alpha}_i$. 

Let's show that $X$ has no small elementary
contraction $\xi\colon
X\to V$ such that $\NE(\xi)\not\subseteq\NE(f)$.
 By contradiction, let $\xi\colon
X\to V$ be such a contraction. Since $f$ is finite on fibers
of $\xi$, these fibers have dimension at most $2$.
If $F_0$ is an irreducible component of a fiber of $\xi$ with $\dim
F_0=2$, then $f(F_0)=Y$, so
$$(f_*)_{|\N(F_0,X)}\colon \N(F_0,X)\to\N(Y)$$
is surjective. This is impossible because $\dim \N(F_0,X)=1$ while
$\rho_Y=3$. 
 
On the other hand, a small contraction can not have only 
$1$-dimensional fibers,
see \cite[Corollary on p.\ 145]{wisn}. Thus we have a contradiction.

Consider now an extremal ray $\beta$ of $\NE(X)$ which is not
contained in $\NE(f)$. Since the contraction of
$\beta$ is not small, Theorem  
\ref{face2} $(ii)$
implies that there is face $\NE(h)$ ($h$ as in \eqref{diag})
of $\NE(X)$, of dimension
$1+\dim\NE(f)$, containing both
$\beta$ and $\NE(f)$. Then 
$f_*(\beta)$ is an extremal ray of $\NE(Y)$,
say $\alpha_1$. This implies that $\beta=\widetilde{\alpha}_1$.

Summing up, we have shown that:
\begin{quote}
\emph{$\widetilde{\alpha}_1$, $\widetilde{\alpha}_2$, 
and $\widetilde{\alpha}_3$ are the only
extremal rays of $\NE(X)$ outside $\NE(f)$}. 
\end{quote}

\smallskip

Let's show that $X$ is a product. 
Applying Remark \ref{nuoto} we get 
the following sequence:
$$\xymatrix{ X \ar[d]_f\ar[r]_{\ph_1}\ar@/^1pc/[rr]^{\zeta} &
  {W_1}\ar[d]\ar[r] & {X_0} \ar[d]^{f_0}\\
Y\ar[r]_{\psi_1} & {\mathbb{F}_1}\ar[r] & {\pr^2}
 }$$
where $\psi_1\colon Y\to \mathbb{F}_1$ is the contraction of $C_1$ and 
$\mathbb{F}_1\to {\pr^2}$ is the contraction of (the image of) $C_2$,
  hence $\zeta\colon X\to X_0$ is  
the contraction of
$\widetilde{\alpha}_1+\widetilde{\alpha}_2$.

Since $\dim X_0\geq 3$, $\NE(f_0)$ is a proper face of $\NE(X_0)$, and
there is at least one extremal ray $\beta$ of $\NE(X_0)$ outside
$\NE(f_0)$. As explained in \ref{lifting}, $\beta$ must be
the image via $\zeta_*$ of some extremal ray $\widetilde{\beta}$
of $\NE(X)$ not lying in
$\NE(\zeta)=\widetilde{\alpha}_1+\widetilde{\alpha}_2$.  
Since $\zeta_*(\NE(f))=\NE(f_0)$ and
$\beta\not\subseteq\NE(f_0)$, it must be
$\widetilde{\beta}\not\subseteq\NE(f)$. Thus the only possibility is 
  $\widetilde{\beta}=\widetilde{\alpha}_3$ and
$\beta=\zeta_*( \widetilde{\alpha}_3)$. 

Let
$$h\colon X_0\longrightarrow W_0$$
be the contraction of $\beta$. Observe that $\dim W_0\geq n-2$, because
$h$ is finite on fibers of $f_0$. 
Our goal is  to show that $X_0\cong \pr^2\times W_0$ and $f_0$ is
the projection, which implies the statement.
The first step is to show that $\dim W_0=n-2$.

For $i=1,2,3$ let $\ph_i\colon X \to W_i$ be the smooth blow-up obtained 
by contracting $\widetilde{\alpha}_i$.
Let $E_i\subset X$ be
the exceptional divisor of $\ph_i$,  $F$ the fiber of $f$ over $C_1\cap
C_3$, and $F'$ the one over $C_2\cap C_3$. 
By Remark \ref{nuoto} 
we have $E_3\cong \pr^1\times F\cong\pr^1\times F'$, so $F\cong
   F'$, let's call it
$F$. 
Thus $E_i\cong\pr^1\times F$, $E_1\cap E_2=\emptyset$, while $E_1\cap
E_3$ and $E_2\cap E_3$ are fibers of $f$.
 
Let $\widetilde{C}_i\subset E_i$ be a curve corresponding to $\pr^1\times\{*\}$.
Then $[\widetilde{C}_i]\in\widetilde{\alpha}_i$, $\widetilde{C}_i\cdot E_i=-1$,
$f_*(\widetilde{C}_i)=C_i$, and $f^*(C_i)=E_i$. This gives: 
$$\widetilde{C}_1\cdot
E_3=\widetilde{C}_3\cdot E_1=\widetilde{C}_2\cdot
E_3=\widetilde{C}_3\cdot E_2=1\quad\text{while} 
\quad \widetilde{C}_1\cdot E_2=\widetilde{C}_2\cdot E_1=0.$$
Finally set $E_3':=\zeta(E_3)$ and  
$C_3':=\zeta_*(\widetilde{C}_3)$, so that $[C_3']\in\beta$.

Observe that $E_3'$ is a divisor in $X_0$ and there are
$m_1,m_2\in\Z_{\geq 0}$ such that
$$\zeta^*(E_3')=E_3+m_1E_1+m_2E_2.$$
Then 
$$0=E_3'\cdot \zeta_*(\widetilde{C}_1)=E_3\cdot
\widetilde{C}_1+m_1E_1\cdot \widetilde{C}_1+ 
m_2 E_2\cdot \widetilde{C}_1=1-m_1$$
so $m_1=1$, and similarly $m_2=1$. Hence
$$E_3'\cdot C_3'=(E_3+E_1+E_2)\cdot \widetilde{C}_3=1.$$
This says that $C_3'$ is a minimal element in the extremal ray
$\beta$, namely that for
every curve $C\subset X_0$ such that $[C]\in\beta$, we have $C\equiv
mC_3'$ with $m\in\Z_{>0}$. 

Moreover
$$K_X=\ph_1^*(K_{W_1})+E_1=\zeta^*(K_{X_0})+E_1+E_2,$$
which gives 
$$ -K_{X_0}\cdot C_3'=\zeta^*(-K_{X_0})\cdot
\widetilde{C}_3=(-K_X+E_1+E_2)\cdot \widetilde{C}_3=3.$$ 
This says that the length $l(\beta)$ of the extremal ray $\beta$, that
is the minimal anticanonical degree of rational curves whose class is
in $\beta$, is $3$. Now \cite[Theorem 1.1]{wisn} yields that for every
non trivial fiber $F$ of $h$ we have
$$\dim F+\dim\Lo(\beta)\geq n+2.$$
On the other hand $f_0$ is finite on fibers of $h$, so $\dim F\leq
2$. Therefore $\dim\Lo(\beta)=n$, every fiber of $h$ has dimension $2$,
and $\dim W_0=n-2$.

Now $f_0$ and $h$ induce a finite morphism $X_0\to \pr^2\times W_0$, let $d$
be its degree. 

Let $w_0\in W_0$ be a general point and $S_0=h^{-1}(w_0)$. Then
$S_0$ is a smooth surface and
$g:=(f_0)_{|S_0}\colon S_0\to\pr^2$ is finite of degree $d$.

Since in $X$ the divisor $E_3$ intersects trasversally $E_1$ and
$E_2$, $\zeta_{|E_3}$ is an isomorphism, so that $E_3'\cong
\pr^1\times F$ and $(f_0)_{|E_3'}$ is the first projection.

Observe that $f_0(E_3')$ is a line $l$ (the line through $p_1$ and $p_2$), 
and $h_{|E_3'}\colon E_3'\to W_0$ is
surjective, so it factors as
$$ E_3'\stackrel{\pi_2}{\longrightarrow} F
\stackrel{\xi}{\longrightarrow} W_0$$
where $\pi_2$ is the second projection and $\xi$ is a finite
morphism. 
We have
$$g^{-1}(l)=(f_0)^{-1}(l)\cap S_0
= E_3'\cap S_0=E_3'\cap h^{-1}(w_0)=
(h_{|E_3'})^{-1}(w_0)\cong\pr^1\times\xi^{-1}(w_0).$$
On the other hand $g^{-1}(l)$ is the support of an ample divisor in
$S_0$ and hence it is connected. This implies that $\xi$ is an
isomorphism
 and that $\widetilde{l}:=g^{-1}(l)\simeq \pr^1\times w_0$.

Then $g^{*}(l)=d\,\widetilde{l}$; on the other hand
$g_{|\widetilde{l}}=(f_0)_{|\widetilde{l}}$ 
is an isomorphism, so $g_*(\widetilde{l})=l$.
Now  we have
$$1=l^2=l\cdot g_*\bigl(\,\widetilde{l}\,\bigr)=g^*(l)\cdot
\widetilde{l}=d\bigl(\,\widetilde{l}\,\bigr)^2,$$ 
so $d=1$ and we are done.
\end{proof}
\section{Fano manifolds with a quasi elementary contraction onto a $3$-fold}
\label{3fold}
In this section we prove Theorem \ref{result} $(ii)$.
Let's consider a  smooth Fano variety $X$ 
of dimension $n\geq 4$, 
and  a quasi
elementary contraction $f\colon X\to Y$
with $\dim Y=3$.

Recall that $Y$ is factorial and has at most isolated canonical
singularities by Lemma \ref{target} $(i)$ and~$(iv)$.

Let $\psi\colon Y\to Z$ be an elementary contraction of $Y$. We
consider all possibilities for $\psi$.

 If $\psi$ is
of type (3,0), then $Z$ is a point and
  $Y$ is Fano with $\rho_Y=1$. 

If $\psi$ is of type (3,1), then $Z\simeq\pr^1$ and
$\rho_Y=2$. In this case the other elementary contraction of $Y$ can
not be again of type (3,1), because non trivial fibers of distinct
elementary contractions can intersect at most in finitely many points,
and $Y$ is factorial.
\begin{claim} 
\label{(3,2)}
If
$\psi\colon Y\to Z$ 
is of type (3,2), then $\psi$ is equidimensional,
and we have the following possibilities: 
\begin{enumerate}[$(i)$]
\item $\rho_Y=2$ and $Z\cong\pr^2$;
\item $\rho_Y=3$ and $Z$ is either $\pr^1\times\pr^1$ or
  $\mathbb{F}_1$;
\item $Z$ is a smooth Del Pezzo surface with $\rho_Z\geq 3$, 
$Y\cong  Z\times\pr^1$, $\psi$ is the projection, and
$X\cong  Z\times W$, where $W$ is a smooth Fano variety of dimension
$n-2$ with a quasi elementary contraction onto $\pr^1$.
\end{enumerate}
\end{claim}
\begin{proof}
We have $\dim Z=2$ and $\psi\circ
f\colon X\to Z$ is quasi elementary (see example~\ref{compo}).
By Theorem~\ref{result} $(i)$,
$\psi\circ f$ is equidimensional and
$Z$ is a smooth Del Pezzo surface. Then $\psi$ must be equidimensional
too.

If $\rho_Z=1,2$, we get the first
two cases. If $\rho_Z\geq 3$, again by Theorem~\ref{result} $(i)$ we see that 
$X\cong
  Z\times W$ where $W$ is a smooth Fano variety of dimension $n-2$,
  and $\psi\circ f$ is the projection onto $Z$.
Then any intermediate
  contraction $Z\times W\to V\to Z$ must be onto a product $V\cong Z\times
  F$, hence we get
  $Y\cong Z\times\pr^1$.
Moreover $f$ induces a contraction $g\colon W\to\pr^1$. Since $f$ is
quasi elementary, it is easy to see that $g$ is quasi elementary too.
\end{proof}
\begin{claim} \label{(2,0)}
If
$\psi$ is of type (2,0), then we have the following possibilities:
\begin{enumerate}[$(i)$]
\item $\rho_Y=2$ and $Y$ has an  elementary contraction of type (3,2);
\item $\rho_Y=3$ and $Y$ has a birational elementary contraction
with fibers of
dimension at most $1$.
\end{enumerate}
\end{claim}
\begin{proof}
The divisor $\Exc(\psi)$ satisfies $\dim\N(\Exc(\psi),Y)=1$, hence
Proposition~\ref{divisor} yields $\rho_Y\leq 3$. However $Y$ has a non
trivial elementary contraction, hence $\rho_Y>1$. Then the statement
follows from Remark \ref{dettaglio}.
\end{proof}
\begin{claim}\label{curry}
Let
$\psi\colon Y\to Z$ be an elementary contraction
which is birational with fibers of dimension at
most $1$. 
Then $\psi$ is the blow-up of a smooth
             curve $C\subset Z_{reg}$, and we have a commutative
             diagram
$$\xymatrix{ X \ar[d]_f \ar[r]^{\ph} & W\ar[d]^g \\
             Y \ar[r]^{\psi} & {Z}}$$
where $W$ is smooth, $\ph$ is the blow-up of a smooth subvariety of
             codimension $2$, $\Exc(\ph)=f^*(\Exc(\psi))$, 
and $g$ is a quasi elementary Mori
             contraction.
\end{claim}
\begin{proof}
Recall that $\psi$ has a lifting as explained in
\ref{lifting}, so we have a diagram as above. 
Since $X$ is Fano, $\ph$ is a Mori contraction. Then 
Theorem \ref{secondo} yields the statement.
\end{proof}
\noindent In particular Claim \ref{curry} implies that \emph{$Y$ has no
small elementary contractions}.
\begin{claim}\label{basta}
In the setting of Claim \ref{curry}, suppose that $Y$ is Fano. Set
$E:=\Exc(\ph)$ and $S:=\Exc(\psi)$. Then
the following statements are equivalent:
\begin{enumerate}[$(i)$]
\item $W$ is not Fano;
\item $Z$ is not Fano;
\item $S\cong\pr^1\times\pr^1$ with normal bundle
$\mathcal{O}_{\pr^1\times\pr^1}(-1,-1)$, and 
there is an extremal ray $\beta$ of $\NE(Y)$,
different from $\NE(\psi)$, such that $S\cdot\beta<0$.  
\end{enumerate}
\end{claim}
\begin{proof}
$(i)\,\Rightarrow \,(ii)\ $
If $W$ is not Fano,
by \cite[Proposition 3.4]{wisn} 
there exists an
  extremal ray $\alpha$ of $\NE(X)$ such that $\alpha\neq\NE(\ph)$
  and $\alpha\cdot E<0$.
Observe that $\alpha$ is not contracted by $f$, because 
 $E\cdot C=0$ for every curve
$C\subset X$ contracted by $f$.

Let $\ph_2\colon X\to W_2$ be the
  contraction of $\alpha$ and let $F_0$ be an irreducible 
component of a non
  trivial fiber of $\ph_2$. Then $F_0\subset E$  and
$f$ is finite on $F_0$, hence $f(F_0)\subseteq S$ and $\dim F_0\leq 2$. 
Moreover if $\dim F_0=2$ then $f(F_0)=S$. This would imply that
$\dim\N(S,Y)=1$, while $\dim\N(S,Y)=2$. Thus
$\dim F_0=1$.

Now \cite[Theorem 1.2]{wisn} yields that $\ph_2$ is a smooth blow-up
with exceptional divisor $E$.
Observe that $S$ is a smooth $\pr^1$-bundle over $\psi(S)$. 
It is not difficult to see that the $\pr^1$-bundle 
$(\ph_2)_{|E}$ induces a second rational fibration on $S$, so that
$S\cong\pr^1\times\pr^1$.
Moreover if
$C\subset E$ is a non trivial fiber of $\ph_2$, then 
$$-1=C\cdot E=C\cdot f^*(S)=f_*(C)\cdot S,$$
and this gives
$\mathcal{N}_{S/Y}\cong\mathcal{O}_{\pr^1\times\pr^1}(-1,-1)$. 
In particular $\psi(S)$ is a curve of anticanonical degree $0$ in $Z$, so $Z$
is not Fano. 

\medskip

$(ii)\,\Rightarrow \,(iii)\ $
Suppose that $Z$ is not Fano. Observe that $Y$ and $Z$ may be singular, however
 they are factorial and $K_Y=\psi^*(K_Z)+S$. Then reasoning as in
 \cite[Proposition 3.4]{wisn} one gets $(iii)$.

\medskip

$(iii)\,\Rightarrow \,(i)\ $
The contraction of $\beta$ is birational with fibers of dimension at
most $1$; let $\ph_2\colon X\to W_2$ be the smooth blow-up given by
 Claim \ref{curry}. Then $E=\Exc(\ph_2)$, and
$\ph_2\neq\ph$.
If $C\subset E$ is a non trivial fiber of
$\ph_2$, then $K_X\cdot C=E\cdot C=1$, which yields $-K_{W}\cdot
\ph_*(C)=0$. Hence $W$ is not Fano.
\end{proof}
\begin{claim}
\label{pollock}
Suppose that $Y$ is Fano, and let $\beta_1$ and
$\beta_2$ be two distinct extremal rays of $\NE(Y)$ with
divisorial loci ${S}_i=\Lo(\beta_i)$. 
If $S_1\cap S_2\neq\emptyset$,
then the contraction $\psi\colon Y\to Z$ of one of the $\beta_i$'s is 
the blow-up of a smooth curve $C\subset Z_{reg}$,
and $Z$ is Fano.
\end{claim}
\begin{proof}
Let $\psi_i\colon Y\to Z_i$ be the contraction of $\beta_i$.
Since $S_1\cap S_2\neq\emptyset$ and $Y$ is factorial, we have
$\dim S_1\cap S_2=1$. Thus
$\psi_1$ and $\psi_2$
can not be both of type (2,0). 

So let's assume that $\psi_1$ is of type (2,1). If $Z_1$ is Fano, we
are done. If $Z_1$ is not Fano, Claim~\ref{basta} yields the existence
of a second
extremal ray $\widetilde{\beta}_1$
of $\NE(Y)$, distinct from $\beta_1$, with $\widetilde{\beta}_1\cdot
{S_1}<0$. Then $\N(S_1,Y)$ is generated by $\beta_1$ and
$\widetilde{\beta}_1$, and no
other extremal ray of $\NE(Y)$ can be contained in $\N(S_1,Y)$. 

Thus $\psi_2$ can not be of type (2,0). 
Let's show that $Z_2$ must be Fano. If not,
by Claim \ref{basta} 
there is another
extremal ray $\widetilde{\beta}_2$
of $\NE(Y)$, distinct from $\beta_2$, with $\widetilde{\beta}_2\cdot
{S_2}<0$. Since $\beta_2$ and $\widetilde{\beta}_2$ are not contained
in $\N(S_1,Y)$, we have $\beta_2\cdot S_1\geq 0$ and
$\widetilde{\beta}_2\cdot S_1\geq 0$.

Now if
$C\subset S_1\cap{S}_2$ is an irreducible curve, we get $C\cdot S_1<0$ because
$\beta_1\cdot S_1<0$ and $\widetilde{\beta}_1\cdot S_1<0$. On the
other hand $C\cdot 
S_1\geq 0$ because ${\beta}_2\cdot S_1\geq 0$ and
$\widetilde{\beta}_2\cdot S_1\geq 0$. Thus we have a contradiction.
\end{proof}
\begin{claim} \label{2,1fano}
In the setting of Claim \ref{curry}, suppose that $Y$ is Fano.
Then
 we are in one of the following situations:
\begin{enumerate}[$(i)$]
\item $W$ and $Z$ are Fano;
\item $\rho_Y\geq 3$ and
$Y$ has another elementary contraction $\widetilde{\psi}$
as in Claim \ref{curry}, such that the corresponding $\widetilde{W}$,
  $\widetilde{Z}$
are Fano;
\item $\rho_Y=3$ and $Y$ has an elementary contraction of type (3,2).
\end{enumerate}
\end{claim}
\begin{proof}
Let's assume 
that we are not in $(i)$, so $W$ and $Z$ are not Fano.
By Claim~\ref{basta} we have $S\cong\pr^1\times\pr^1$, and there is a
second extremal ray $\beta$ of $\NE(Y)$ with $\beta\cdot S<0$.

There exists an extremal ray $\widetilde{\beta}$ of $\NE(Y)$ such that
$S\cdot \widetilde{\beta}>0$, let $\widetilde{\psi}\colon Y\to
\widetilde{Z}$ be its contraction.

Clearly $\widetilde{\beta}$ is distinct from $\NE(\psi)$ and $\beta$,
and the elements of three distinct extremal rays must by linearly
independent in $\mathcal{N}_1(Y)$. Hence $\rho_Y\geq 3$ and
$\widetilde{\beta}\cap\N(S,Y)=\{0\}$, which implies that
$\widetilde{\psi}$ is finite on $S$. On the other hand $S$ must
intersect every non trivial fiber of $\widetilde{\psi}$, hence
$\widetilde{\psi}$ has fibers of dimension at most $1$.

If $\widetilde{\psi}$ is of fiber type, then it is of type (3,2).
Let's consider
$$\bigl(\widetilde{\psi}_*\bigr)_{|\N(S,Y)}\colon\N(S,Y)\la\N(\widetilde{Z}).$$
Since $\ker\widetilde{\psi}_*$ is the line spanned by $\widetilde{\beta}$ in
$\N(Y)$, we have $\ker\widetilde{\psi}_*\cap\N(S,Y)=\{0\}$ and
$(\widetilde{\psi}_*)_{|\N(S,Y)}$ is injective. On the other hand
$\widetilde{\psi}(S)=\widetilde{Z}$, hence
$(\widetilde{\psi}_*)_{|\N(S,Y)}$ is surjective too. This gives
$\rho_{\widetilde{Z}}=2$ and $\rho_Y=3$, and thus $(iii)$.

Suppose that  $\widetilde{\psi}$ is birational. Then 
$\Exc(\widetilde{\psi})\cap\Exc(\psi)\neq\emptyset$, and
Claim \ref{pollock} implies that $\widetilde{Z}$ is Fano. Finally also 
$\widetilde{W}$ is Fano by Claim \ref{basta}, and we get $(ii)$.
\end{proof}
\begin{claim}
\label{m}
Let $\pi\colon Y\to T$ be a contraction onto a surface.
Let $\alpha_1,\dotsc,\alpha_m$ be the  extremal rays of $\NE(\pi)$
and  $\psi_i\colon Y\to Z_i$ the contraction of $\alpha_i$.
Then 
$\Exc(\psi_i)\cap\Exc(\psi_j)\neq\emptyset$ for some $i\neq j$.
\end{claim}
\begin{proof}
Set $S_i:=\Exc(\psi_i)$, and 
assume by contradiction that  $S_i\cap S_j=\emptyset$ for every $i\neq
j$. Then each $\psi_i$ is birational and 
$S_i\cdot\alpha_i<0$ for every $i$
by Remark \ref{divisorial}; moreover $S_i\cdot\alpha_j=0$ for every
$i\neq j$. 

Let $C\subset Y$ be an irreducible curve which is contracted by $\pi$
but not contained in $S_1\cup\cdots\cup S_m$. Then 
$$C\equiv\sum_{i=1}^m \lambda_i l_i\ \text{ where $[l_i]\in\alpha_i$
 and $\lambda_i\in\Q_{\,\geq 0}$ for every $i=1,\dotsc,m$.}$$ 
So $0\leq C\cdot S_i=\lambda_i(l_i\cdot S_i)$ implies $\lambda_i=0$
for every $i=1,\dotsc,m$, a contradiction.
\end{proof}
\begin{proof}[Proof of Theorem \ref{result} $(ii)$]
Let's suppose that 
 $\rho_Y\geq 4$. By Claims \ref{(3,2)}, \ref{(2,0)}, and 
\ref{curry}, 
 any elementary contraction of $Y$
is either a smooth blow-up or a $\pr^1$-bundle. This implies that $Y$
is Fano (but possibly singular) by Remark \ref{fano}.

 Suppose now that $Y$ is smooth and that $\rho_Y\geq 6$. Then $Y\cong
S\times\pr^1$ with $S$ a Del Pezzo surface by \cite[Theorem
2]{morimukai}. Thus $X\to S$ is a quasi elementary contraction
(see example \ref{compo}), and
applying Theorem \ref{result} $(i)$ 
as in the proof of Claim \ref{(3,2)} we easily get the
statement.

Thus we are left to prove that if $\rho_Y\geq 4$, then $Y$ must be
smooth.

\medskip

By contradiction, assume that $Y$ is singular and $\rho_Y\geq 4$. By
Claims \ref{(3,2)}, \ref{(2,0)}, and \ref{2,1fano}, we
can construct a sequence 
$$\xymatrix{
{X=X_{\rho_Y}} \ar[d]_f \ar[r]
& {X_{\rho_Y-1}}
\ar[r]
\ar[d]
& {\cdots}\ar[r]& {X_4} \ar[d]^{f_4}
\\
{Y=Y_{\rho_Y}} \ar[r]^{\psi_{\rho_Y}} & {Y_{\rho_Y-1}}
\ar[r]
& {\cdots}\ar[r]^{\psi_5}& {Y_4} 
}$$
where for every $i=4,\dotsc,\rho_Y$ 
$\,X_i$ is smooth and Fano, $f_i$ is a quasi elementary
contraction, $Y_i$ is Fano with $\rho_{Y_i}=i$, and $\psi_i$ is the blow-up of a
smooth curve contained in the smooth locus of $Y_{i-1}$. In particular
each $Y_i$ is singular, because $Y$ is.

Since $Y_4$ is singular and $\rho_{Y_4}=4$, 
Claim \ref{(3,2)} implies that $Y_4$ has no
elementary contraction of fiber type.
\begin{claim}
\label{final}
$Y_4$ has no contraction onto a surface.
\end{claim}
\begin{proof}
By contradiction let $\pi\colon Y_4\to T$ be such a contraction. 
Then $\pi$ can not be elementary, so $\rho_T\leq 2$.

The cone $\NE(\pi)$ contains $m\geq 2$ extremal rays, whose contractions
are birational. Call $S_1,\dotsc,S_m$ their exceptional loci.
By Claim \ref{m} they can not be all disjoint, so we can assume that $S_1\cap
S_2\neq \emptyset$. 

Using Claim \ref{pollock}, we can assume that the elementary
contraction $\psi_4\colon Y_4\to Y_3$ with exceptional locus $S_1$ 
is the blow-up of a smooth curve $C_3\subset(Y_3)_{reg}$, 
and that $Y_3$ is Fano. We observe that $Y_3$ is singular because
$Y_4$ is.
As before we get
$$\xymatrix{ {X_4} \ar[d]_{f_4}\ar[r]&
{X_3}\ar[d]^{f_3}\\
{Y_4} \ar[r]^{\psi_4} & {Y_3}
}$$
where $X_3$ is smooth and Fano and $f_3$ is
a quasi elementary contraction.
Moreover
$\pi$ induces a contraction $\pi'\colon Y_3\to T$.

If $\pi'$ is elementary, then 
$\rho_T=2$, $\pi'$ is equidimensional, and 
 $T$ is either $\pr^1\times\pr^1$ or
$\mathbb{F}_1$ by Claim~\ref{(3,2)}.
Hence $\pi'$ is a conic bundle (see \cite[\S 1]{sarkisov}
for the definition and properties of conic bundles).
Let $\Delta_{\pi'}\subset T$ be
the discriminant locus of $\pi'$, which is non empty because $Y_3$ is
singular. 

We recall that
 $C_3\subset Y_3$ can
not intersect any curve of anticanonical degree $1$ by Remark
\ref{wellknown}.
Thus $C_3$ can not be a component of
a reducible fiber of $\pi'$.

If $C_3$ is an irreducible fiber, let
$p=\pi'(C_3)\in T$ and call $T'$ the blow-up of $T$ in $p$: then $Y_4$
has an elementary contraction onto $T'$, which
is impossible.

Hence $\pi'(C_3)$ is a curve in $T$, and
it is disjoint from
$\Delta_{\pi'}$ because
$C_3$ can not
intersect singular fibers of $\pi'$. 

If $T\cong\pr^1\times\pr^1$,
the only possibility is that
$\Delta_{\pi'}$ is a union of fibers of a projection
$\pr^1\times\pr^1\to \pr^1$, hence a disjoint union of 
smooth rational curves. But this is impossible, see for instance
\cite[Lemma 5.3]{prok}.

Consider now the case where $T\cong\mathbb{F}_1$. 
Then we have a contraction $Y_3\to\mathbb{F}_1\to\pr^2$, which has a
second factorization $Y_3\to Y_2\to\pr^2$. It is not difficult to see
that $\psi_3 \colon
Y_3\to Y_2$ is birational with fibers of dimension $\leq 1$, so by
Claim \ref{curry} it is
again a blow-up of a smooth curve $C_2\subset(Y_2)_{reg}$, and we have
a diagram:
$$\xymatrix{ {X_3} \ar[d]_{f_3}\ar[r]&
{X_2}\ar[d]^{f_2}&\\
{Y_3} \ar[r]^{\psi_3} & {Y_2} \ar[r]^{\xi}& {\pr^2}
}$$
where $X_2$ is smooth, $Y_2$ is singular with $\rho_{Y_2}=2$, 
and $f_2$ is a quasi elementary Mori
contraction.

The only curve in $Y_2$ which could have non positive anticanonical
degree is $C_2$. Thus $\xi\colon Y_2\to \pr^2$ is a conic bundle 
with non empty
 discriminant locus $\Delta_{\xi}\subset\pr^2$.
If  $\xi(C_2)$ is a curve in $\pr^2$, then
$C_2$ must intersect some singular fiber,
which is again impossible by Remark \ref{wellknown}.
Thus $C_2$ is a smooth fiber of $\xi$, $Y_2$ is Fano, and $X_2$ is
Fano too by Claim \ref{basta}.

Let's consider the other elementary contraction $\eta\colon Y_2\to Z$
of $Y_2$.

If $\eta$ is again of type (3,2), reasoning as above we get that
$\eta$ must contract $C_2$, a contradiction because
$\NE(\eta)\cap\NE(\xi)=\{0\}$.

If $\eta$ is of type (3,1), it is a fibration in Del Pezzo
surfaces over $\pr^1$. 
Then $Y_2$ is a finite cover of $\pr^1\times\pr^2$, and $C_2$ is the
inverse image of
$\pr^1\times\{pt\}$. This implies that 
$Y_3$ is a finite cover of
$\pr^1\times\mathbb{F}_1$, which gives a surjective morphism 
$Y_3\to\pr^1\times\pr^1$. Taking the Stein factorization we get
an elementary contraction $Y_3\to T'$
of type (3,2), where $T'$ is a
finite cover of $\pr^1\times\pr^1$. By Claim \ref{(3,2)}
$T'$ can be $\pr^1\times\pr^1$ or $\mathbb{F}_1$, but
$T'$ has two distinct
fibrations, thus
$T'\cong\pr^1\times\pr^1$. We have already excluded this
possibility.

Therefore $\eta$ is birational, let $E$ be its exceptional
divisor. Since $E\cdot\NE(\eta)<0$ by Remark \ref{divisorial}, it must
be $E\cdot\NE(\xi)>0$, hence $\xi(E)=\pr^2$ and $E$ intersects
$C_2$. In particular $E$ can not be covered by curves of anticanonical
degree $1$, thus $\eta$ is of type (2,0). Moreover $E\cdot C_2\geq 2$,
because $\xi$ has singular fibers.

The composite
contraction $Y_3\to Z$ has a second factorization
$$Y_3\stackrel{\sigma}{\la}\widetilde{Y}_2\stackrel{\chi}{\la} Z.$$ 
It is not difficult to see that
$\sigma$ has
exceptional locus $\psi_3^*(E)$ and 
is the blow-up of a smooth curve
$\widetilde{C}_2\subset(\widetilde{Y}_2)_{reg}$.
Since the image of $\psi_3^*(E)$ in $Z$ is $z_0:=\eta(E)$, 
$\widetilde{C}_2$ is
contracted by $\chi$. Then we see that $\chi$ is again a Mori
contraction of type (2,1) with exceptional divisor
$\sigma(\Exc(\psi_3))$, and $\widetilde{Y}_2$ is Fano.
In particular $\chi$ is again a blow-up of a smooth curve contained in
$Z_{reg}$, so that $z_0$ is a smooth point. Moreover  $\psi_3^*(E)$ is
contained in the smooth locus of $Y_3$, so that $E\subset (Y_2)_{reg}$.

Therefore $\eta$ is just the blow-up of $z_0$ and $E\cong\pr^2$.
If $E$ intersects $C_2$ in at least
two distinct points, take $l$ a line in
$E$ through these two points. Let $\widetilde{l}$ be the proper
transform of $l$ in $Y_3$. Then $-K_{Y_2}\cdot l=2$ and
$\widetilde{l}\cdot \Exc(\psi_3)\geq 2$, so
 by Remark~\ref{wellknown} we get $-K_{Y_3}\cdot
\widetilde{l}\leq 0$, a contradiction.

If $E$ intersects $C_2$ in a single non reduced point $y_0$, similarly as
before take $l$ a line in
$E\cong\pr^2$ through $y_0$. Then the schematic intersection
$\psi_3^*(E)\cap \Exc(\psi_3)$ in non reduced along the fiber of
$\psi_3$ over $y_0$, thus again $\widetilde{l}\cdot \Exc(\psi_3)\geq 2$
gives a contradiction.
This concludes the case where $\rho_T=2$.

\medskip

We still have to exclude the case where
$\rho_T=1$ and
$Y_3$  has no elementary contractions of type (3,2).
Reasoning as for $Y_4$, we see that one of  
the two extremal rays of $\NE(\pi')$ yields a blow-up
$\psi_3\colon
Y_3\to Z_2$ of a smooth curve $C'\subset(Z_2)_{reg}$, and
$Z_2$ is Fano with $\rho_{Z_2}=2$. 
Now $\pi'$ yields an elementary contraction $\pi''\colon Z_2\to
T$. Claim \ref{(3,2)} gives 
$T\cong\pr^2$, and $\Delta_{\pi''}$ is
non empty. As before we easily get a contradiction.
\end{proof}
Using Claims \ref{final}, \ref{(2,0)}, and \ref{2,1fano}, we get a 
sequence
$$\xymatrix{
 {Y_4} \ar[r]^{\psi_4} & {Y_3} \ar[r]^{\psi_3} & {Y_2} \ar[r]^{\psi_2} & 
{Y_1} 
}$$
where each $Y_i$ Fano 
and each $\psi_{i}$ is the blow-up of a smooth curve
$C_{i-1}\subset (Y_{i-1})_{reg}$. In particular  $Y_1$
is factorial with isolated canonical singularities, singular, and
$\rho_{Y_1}=1$.

We first show that $Y_1$ must have terminal singularities, using the
following Lemma.
\begin{lemma}
\label{canonical}
Let $Z$ be a $3$-dimensional  $\Q$-factorial
projective variety, with isolated canonical
singularities and $K_Z$ Cartier. Suppose that $Z$ is Fano
with $\rho_Z=1$, and that the singularities of $Z$ are not terminal. 
Then 
one of the following occurs:
\begin{enumerate}[$(i)$]
\item 
$Z$ contains a $1$-dimensional family of curves of anticanonical
  degree $1$ passing through a singular point;
\item $Z$ is covered by a family 
of curves of anticanonical degree $\leq 2$ passing
  through a singular point. 
\end{enumerate}
\end{lemma}
\noindent We postpone the proof of Lemma
\ref{canonical} and carry on with the proof of Theorem~\ref{result} $(ii)$. 

Suppose that $Y_1$ has at least one non-terminal singular point:
then Lemma~\ref{canonical} applies to $Y_1$. 
If $(i)$ holds, let $S\subset Y_1$ be a surface covered by curves of
anticanonical degree $1$. 
Since $\rho_{Y_1}=1$, $S$ is ample, and
$C_1\cap S\neq\emptyset$. Observe that even if $C_1\subset S$, $C_1$
does not contain any singular point, hence it can not be a member of
the family given by $(i)$. Thus $C_1$ intersects some curve of
anticanonical degree $1$,
which is impossible by Remark \ref{wellknown}.

Suppose now that $(ii)$ holds for $Y_1$. If $C_1$ is a component of some
reducible curve $l_1$ of the family,
it must be $l_1=C_1\cup C_1'$ with $-K_{Y_1}\cdot C_1=-K_{Y_1}\cdot
C_1'=1$, which gives again a contradiction. 
Again $C_1$ can not be a member of the family, because it does not
contain singular points.
Hence $C_1$ is not
contained in any member of the family; let $T$ be an irreducible
 surface containing $C_1$ such that through every point of $C_1$ there
 is a curve of 
anticanonical degree $\leq
2$ contained in $T$. Let $\widetilde{T}$ be the proper transform of
$T$ in $Y_2$.  Then through every point of $\widetilde{T}\cap
\Exc(\psi_2)$ there is a curve of anticanonical degree $1$
contained in
$\widetilde{T}$ 
(see Remark~\ref{wellknown}).

Consider $C_2\subset (Y_2)_{reg}$. If $\psi_2(C_2)$ is a point, then $C_2$
must intersect some curve of anticanonical degree $1$ contained in
$\widetilde{T}$. On the other hand if $\psi_2(C_2)$ is a curve, then
it must intersect $T$, thus $C_2$ must intersect
$\psi_2^{-1}(T)=\widetilde{T}\cup \Exc(\psi_2)$. In any case $C_2$ will
intersect some curve of anticanonical degree $1$, which gives a
contradiction.

\medskip

Hence $Y_1$ has terminal singularities, and the same holds for each
$Y_i$. Consider in particular~$Y_4$. By~\cite{namikawa} $\,Y_4$ has
a smoothing, that is 
an integral complex space $\mathcal{Y}$,
with a projective flat morphism $\mathcal{Y}\to\Delta$ onto the
complex unit
disc, such that $Y_4$ is the fiber over $0$ while the fiber
$Y_t$ over $t\neq 0$ is a smooth Fano threefold. 
It is proven in   \cite{smoothings} that 
$\Pic Y_4 \cong\Pic Y_t$,
in particular $\rho_{Y_t}=4$. 
Then we know by Mori and Mukai's classification
that $Y_t$ has a conic bundle structure (not
necessarily elementary), see
\cite[Theorem on p.\ 141]{fanoEMS}. Our goal is
 to deduce from this that $Y_4$ must have
 a contraction  onto a surface, contradicting Claim \ref{final}.

For this we need the following Lemma, 
 based on \cite{smoothings}.
\begin{lemma}
\label{conodiMori}
Let $Z$ be a $3$-dimensional factorial Fano variety with terminal singularities
and $\mathcal{Z}\to\Delta$ a smoothing. Consider the inclusions
$i_t\colon Z_t\hookrightarrow
\mathcal{Z}$ and $i_0\colon Z_0\hookrightarrow \mathcal{Z}$.

Then the push-forwards $(i_t)_*$ and $(i_0)_*$
induce bijections among the cones
$$\NE(Z_t),\quad \NE(\mathcal{Z}/\Delta),
\quad \text{and}\quad \NE(Z_0),$$
and every contraction of $Z_t$ or of $Z_0$ is the restriction of the
contraction of the corresponding face of $\NE(\mathcal{Z}/\Delta)$.
\end{lemma}
\noindent 
By Lemma \ref{conodiMori}, the conic bundle on $Y_t$ induces a 
 contraction $\mathcal{Y}\to\mathcal{T}\to\Delta$.
This restricts to a contraction $\pi\colon
Y_4\to T$ onto a surface, which contradicts Claim \ref{final}.
This concludes the proof of Theorem \ref{result}.
\end{proof}
\begin{proof}[Proof of Lemma \ref{canonical}]
Let
$\psi\colon Z'\to Z$ be a partial crepant resolution such that $Z'$ has
terminal and factorial singularities (see \cite[\S 6.3]{KollarMori}). Hence
 $K_{Z'}=\psi^*(K_Z)$,  $-K_{Z'}$ is
nef and big, and $\rho_{Z'}>1$ because $\psi$ is not an isomorphism.

Since $Z$ is $\Q$-factorial, for every non terminal point $p\in Z$ the
inverse image $\psi^{-1}(p)$ has pure dimension $2$, therefore
$\Exc(\psi)$ is a divisor. Moreover
any irreducible
curve contained in  $\Exc(\psi)$ has anticanonical degree zero.

There is at least one elementary Mori contraction $f\colon Z'\to
W$. Observe that $f$ must be finite on $\Exc(\psi)$, because any
curve contracted by $f$ has positive anticanonical degree.
Hence any fiber $F$ of $f$ such that $F\cap\Exc(\psi)\neq\emptyset$
has dimension at most $1$.

Suppose that $f$ is of fiber type. Then it must be of type (3,2) and
\cite[Theorem 7]{cut} 
says that $W$ is a smooth 
surface and $f$ is a conic bundle. Moreover if $E\subseteq\Exc(\psi)$ is
an irreducible component, then $f(E)=W$, so that every fiber of $f$
intersects $E$. Then the fibers of $f$ give a covering family of
curves of anticanonical degree $\leq 2$ in $Z$, all passing through
the singular point $\psi(E)$, and we get~$(ii)$.

Suppose that $f$ is birational. Since $Z'$ is Gorenstein, $f$ is
divisorial; set $D:=\Exc(f)$. We claim that $D$ can not be disjoint
from $\Exc(\psi)$. In fact if so, $\psi(D)$ would be a non nef Cartier
divisor in $Z$, which is impossible because $\rho_Z=1$. 

Hence $f$ must be of type (2,1) and by \cite[Theorem 4]{cut} we have
$-K_{Z'}\cdot l=1$
for the
general fiber $l$ of $f$. Again if $E$ is an irreducible component of
$\Exc(\psi)$ intersecting $D$, every fiber $l$ of $f$ must intersect $E$.
So we get a one-dimensional family of
curves of anticanonical degree $1$ in $Z$, passing through the
singular point $\psi(E)$. 
\end{proof}
\begin{proof}[Proof of Lemma \ref{conodiMori}]
By \cite[Proposition 1.1]{smoothings} $\mathcal{Z}$ has at most
isolated terminal factorial singularities at the singular points of
$Z_0$.

We refer to \cite[\S 0-1]{KMM} for the notation in the relative
situation. Observe that for a projective morphism in the analytic
category the standard results of MMP hold, see \cite[Example
2.17]{KollarMori} and references therein. In particular since each
fiber of $\mathcal{Z}\to\Delta$ is Fano, $\NE(\mathcal{Z}/\Delta)$ is
closed and polyhedral by the relative version of the Cone Theorem.

We first observe that the linear maps
$$(i_t)_*\colon\N(Z_t)\la\N(\mathcal{Z}/\Delta)\quad
\text{and}\quad (i_0)_*\colon\N(Z_0)\la\N(\mathcal{Z}/\Delta)$$
are isomorphisms. In fact they are dual to the restrictions
$$\Pic(\mathcal{Z})\otimes{\R}\la\Pic(Z_t)\otimes{\R}\quad\text{and}
\quad \Pic(\mathcal{Z})\otimes{\R}\la\Pic(Z_0)\otimes{\R},$$
which are isomorphism by \cite[Theorem 1.4]{smoothings}. Moreover we
 have $(i_t)_*\NE(Z_t)\subseteq \NE(\mathcal{Z}/\Delta)$ and
$(i_0)_* \NE(Z_0)\subseteq \NE(\mathcal{Z}/\Delta)$.

Up to shrinking $\Delta$, we can
assume that for every extremal ray $\alpha$ of
$\NE(\mathcal{Z}/\Delta)$ either $\Lo(\alpha)$ dominates $\Delta$, or
$\Lo(\alpha)$ is contained in $Z_0$. Let's show that the second case
can not happen.

Assume by contradiction that $\Lo(\alpha)$ is contained in $Z_0$
and consider the contraction $\ph$ of $\alpha$:
$$\xymatrix{ {\mathcal{Z}} \ar[d]\ar[r]^{\ph} & {\mathcal{W}}\ar[dl] \\
{\Delta} & }
$$
Then $\ph$ is an isomorphism outside the central fiber, and restricts
to a contraction $\ph_0\colon Z_0\to W_0$ on the central
fiber. The morphism $\mathcal{W}\to\Delta$ is projective and the
general fiber has dimension $3$, hence $\dim W_0=3$ and $\ph_0$ is
birational.

This means that $E=\Exc(\ph)=\Exc(\ph_0)$ has dimension at most $2$,
namely $\ph$ is a small contraction of $\mathcal{Z}$ and it must be of
type (2,0). Hence
$\mathcal{W}$
is not factorial at $\ph(E)$.

Since $(i_0)_*$ is injective, $\ph_0$
is an elementary contraction of type (2,0) of $Z_0$, and
$E$ is irreducible. Moreover $W_0$ has terminal singularities and is
Fano.

If $W_0$ were Gorenstein, $\mathcal{W}$ should be factorial again by
\cite[Proposition 1.1]{smoothings}, which is not the case: so
$K_{W_0}$ is not Cartier. The possibilities for $\ph_0$ are given in
\cite[Theorem 5]{cut}, and the only case where 
$K_{W_0}$ is not Cartier is $E\cong\pr^2$ with normal bundle
$\mathcal{N}_{E/Z_0}\cong\mathcal{O}_{\pr^2}(-2)$. 

Since $E$ is a smooth prime divisor in the factorial variety $Z_0$, it is
contained in $(Z_0)_{reg}$ and hence in $\mathcal{Z}_{reg}$. Now
\cite[Theorem 2.1]{kawsmall} yields
$\mathcal{N}_{E/\mathcal{Z}}\cong\mathcal{O}_{\pr^2}(-1)^{\oplus 2}$,
a contradiction. 

Hence $\Lo(\alpha)$ dominates $\Delta$, which means that $\alpha$ is
contained in both cones $(i_t)_*\NE(Z_t)$
and $(i_0)_* \NE(Z_0)$. Repeating this for every extremal ray of 
$\NE(\mathcal{Z}/\Delta)$, we get the statement.
\end{proof}
\section{Applications and examples} 
\label{ea}
In this section we prove the Corollaries stated in the introduction,
and some other applications. We also give some related examples.

Suppose that
 $X$ is a smooth Fano variety with  a quasi
elementary contraction $f\colon X\to Y$. If $X$ has
 other suitable contractions, one
can use Theorems \ref{face} and \ref{face2} to get a quasi elementary
contraction $h\colon X\to Z$ with $\dim Z\leq 3$, and then apply
Theorem~\ref{result}. Corollaries \ref{appl}, \ref{cor1}, and
\ref{cor2} are obtained in this way.
\begin{proof}[Proof of Corollary \ref{appl}]
Recall that $\rho_X-\rho_{Y_i}\leq\rho_{F_i}$ because $f_i$ is quasi
elementary. 
Suppose that $\dim F_1+\dim F_2\geq n-2$.
By Theorem \ref{face2}
$(i)$ and Remark \ref{dim}, 
there exists a contraction $h\colon X\to Z$ where
$\dim Z\leq n-(\dim F_1+\dim
F_2)$ and
$$\rho_X-\rho_Z=(\rho_X-\rho_{Y_1})+(\rho_X-\rho_{Y_2})
\leq\rho_{F_1}+\rho_{F_2}.$$
This immediately gives $\dim F_1+\dim F_2\leq n$, and if equality
holds $Z$ is a point so $\rho_Z=0$.
Moreover $\dim F_1+\dim F_2=n-1$ implies $Z\cong\pr^1$ and
$\rho_Z=1$. 

Let's notice that $h$ is quasi elementary too. In fact let $G$ be a 
general fiber of $h$, then $G$ contains general fibers of $f_1$ and $f_2$.
Since both $f_i$'s are quasi elementary, $\N(G,X)$ contains both 
$\ker (f_1)_*$ and $\ker(f_2)_*$. On the other hand we have
$$\N(G,X)\subseteq\ker h_*=\ker(f_1)_*+\ker(f_2)_*$$
as shown in the proof of Theorem \ref{face2}, so $\N(G,X)=\ker h_*$
and $h$ is quasi elementary.

Hence if $\dim F_1+\dim F_2=n-2$ then $Z$ is a Del Pezzo surface by
Theorem~\ref{result} $(i)$, so that $\rho_Z\leq 9$.

Finally suppose that $\dim F_1+\dim F_2=n-3$  and
that $f_2$ is elementary. Similarly to the previous case
Theorem \ref{face} $(i)$ gives a quasi
elementary contraction $h\colon X\to Z$ with $Z\leq 3$ and $\rho_X\leq
\rho_{F_1}+\rho_Z+1$. 
If $\dim Z\leq 2$ we proceed as before. If $\dim Z=3$ then $\rho_Z\leq
10$ by Theorem \ref{result} $(ii)$, so we are done.
\end{proof}
\begin{remark}\label{elem}
In the statement of Corollary \ref{appl} one can replace $\rho_{F_i}$
by $\dim\ker(f_i)_*$, which gives a better bound for instance when
$f_i$ is elementary. Similarly in the following corollaries.
\end{remark}
\begin{corollary}
\label{cor1}
Let $X$ be a smooth Fano variety and $f\colon X\to Y$
a quasi elementary contraction of fiber type
with $\dim Y\geq 3$ and general
fiber $F$.

Let $\ph\colon X\to W$ be an elementary contraction such
that $\NE(\ph)\cap\NE(f)=\{0\}$.
Then every fiber of $\ph$ has dimension at most $\dim Y$, moreover:
\begin{enumerate}[$\bullet$]
\item
if $\ph$ has a fiber of dimension $\dim Y$, then $\rho_X\leq 1+\rho_F$ and
$\rho_Y=1$;
\item
if $\ph$ has a fiber of dimension $\dim Y-1$, then $\rho_X\leq 3+\rho_F$ and
$\rho_Y\leq 3$.
\end{enumerate}
\end{corollary}
\noindent We state the Corollary in this form for completeness, however  
let us notice that only the last statement is really new.
\begin{proof}
Recall that $\rho_X\leq\rho_Y+\rho_F$ because $f$ is quasi elementary.
Since $f$ is finite on fibers
of $\ph$, they have dimension at most $\dim Y$. If there is a fiber with
the same dimension as $Y$, let $F_0$ be an irreducible component with
$\dim F_0=\dim Y$. Then $f(F_0)=Y$, so
$$(f_*)_{|\N(F_0,X)}\colon\N(F_0,X)\la\N(Y)$$
is surjective and $\rho_Y\leq\dim \N(F_0,X)=1$. 

If $\ph$ has a fiber of
dimension $\dim Y-1$, let $F_0$ be an irreducible component with
$\dim F_0=\dim Y-1$. Then $f(F_0)$ is a prime divisor in $Y$ with 
$\dim\N(f(F_0),Y)=1$, so Proposition \ref{divisor} implies that
$\rho_Y\leq 3$.
\end{proof}
\begin{corollary}
\label{cor2}
Let $X$ be a smooth Fano variety and $f\colon X\to Y$
a quasi elementary contraction of fiber type with $\dim Y\geq 3$ and general
fiber $F$.

Let $\ph\colon X\to W$ be a divisorial elementary contraction such
that $\NE(\ph)\cap\NE(f)=\{0\}$ and $f(\Exc(\ph))=Y$.
\begin{enumerate}[$\bullet$]
\item
If the general fiber of $\ph_{|\Exc(\ph)}$ has dimension $\dim Y-2$, then
$\rho_X\leq \rho_F+10$.
\item
If $f$ is elementary and the general fiber of $\ph_{|\Exc(\ph)}$ has dimension
$\dim Y-3$, then $\rho_X\leq 12$.
\end{enumerate} 
\end{corollary}
\begin{proof}
The proof is very similar to that of Corollary \ref{appl}. We apply
Theorem \ref{face2} $(iii)$ and Theorem~\ref{result} $(i)$ in the first case,
Theorem \ref{face} $(iii)$ and Theorem \ref{result} $(ii)$ 
in the second case.  
\end{proof}
\begin{proof}[Proof of Corollary \ref{4fold}] 
The first two statements are a direct consequence of
Theorem \ref{result} $(i)$. For the last statement, suppose that 
$\rho_X\geq 7$ and that $f\colon X\to Y$ is an elementary
contraction with $\dim Y=3$. Then $\rho_Y\geq 6$, so Theorem
\ref{result} $(ii)$ says that $X\cong S\times S'$ where $S$, $S'$ are
Del Pezzo surfaces and $S'$ has an elementary contraction onto
$\pr^1$. Then $S'\cong\pr^1\times\pr^1$ or $S'\cong\mathbb{F}_1$.
\end{proof}
\begin{proof}[Proof of Corollary \ref{5fold}]
The first two statements follow from Corollary \ref{appl} and 
Theorem \ref{result}.
Let
$f\colon X\to Y$ be an elementary
contraction with $\dim Y=4$, and $\ph\colon X\to W$ another elementary
contraction. If $\ph$ has a fiber of dimension at least $3$, then
$\rho_X\leq 4$ by Corollary~\ref{cor1}. In particular this holds if
$\ph$ is of type (3,0), (4,0), or (4,1). Finally suppose that 
 $f(\Exc(\ph))=Y$. If $\ph$ is of fiber type,
we have $\rho_X\leq 12$ by the previous part.
If $\ph$ is birational, then it must be divisorial, 
so $\rho_X\leq 12$ by Corollary
\ref{cor2}.
\end{proof}
Finally we give an application in the spirit of \ref{div}.
\begin{corollary}
Let $X$ be a smooth Fano variety and $f\colon X\to Y$ a non trivial 
quasi elementary contraction 
of fiber type with general fiber $F$. 

Let $D\subset X$ be a prime divisor such that $\dim\N(D,X)=2$.
Then one of the following occurs.
\begin{enumerate}[$(i)$]
\item $\rho_X= 3$, $Y\simeq\pr^1$, and $D$ is a fiber of $f$;
\item $\rho_X\leq 4$, $\rho_Y\leq 3$, and $f$ is elementary;
\item $\rho_X\leq 10$, $X\cong F \times Y$,
$Y$ is a Del Pezzo surface, $\rho_F=1$,
  and $\dim f(D)=1$;
\item $\rho_X\leq 1+\rho_F$, $\rho_Y=1$, and $f(D)=Y$.
\end{enumerate}
\end{corollary}
\begin{proof}
Suppose that $f(D)\subsetneq Y$. Then  Lemma \ref{nonequilocus} says
that $D=f^{-1}(f(D))$; in particular $D$ contains a fiber $F$, so
$$\ker f_*=\N(F,X)\subseteq\N(D,X)$$
and $\rho_X-\rho_Y\leq 2$. If $\rho_X-\rho_Y=2$, then $\N(D,X)=\ker
f_*$, so $D$ is a fiber and $Y\cong\pr^1$. Thus we are in $(i)$.

If $\rho_X-\rho_Y=1$, then $f$ is elementary and $D$ is not a fiber,
hence $\dim Y\geq 2$. If $\rho_Y\leq 3$, we are in $(ii)$. 

Suppose that $\rho_Y\geq 4$. We have $\dim\N(f(D),Y)=1$, so
Proposition \ref{divisor} yields that $\dim Y=2$. Thus Theorem
\ref{result} $(i)$
gives $(iii)$.

Assume now that $f(D)=Y$. Then
the restriction
$$(f_*)_{|\N(D,X)}\colon\N(D,X)\la\N(Y)$$
is surjective, so $\rho_Y\leq 2$. If equality holds, then
$\N(D,X)\cap\ker f_*=\{0\}$ and hence $\N(D,X)\cap\NE(f)=\{0\}$. This
  implies that $f_{|D}$ is finite, so $f$ has $1$-dimensional
  fibers. Then $f$ is elementary and $\rho_X=3$, so we are again in
  $(ii)$. 
Finally if $\rho_Y=1$ we get $(iv)$.
\end{proof}
\begin{example}[(Elementary contractions over surfaces)]
\label{firstex}
It is not difficult to write down examples of smooth
Fano varieties of dimension $n\geq 3$
which are not products,
 but have an elementary contraction of fiber type over
$\pr^2$, $\pr^1\times\pr^1$, or $\mathbb{F}_1$.
Thus the condition $\rho_Y\geq 3$ in the second part of
Theorem~\ref{result} $(i)$ is necessary. 

For instance one can consider the $\pr^{n-2}$-bundles: 
$$\pr_{\pr^2}(\mathcal{O}^{\oplus
  (n-2)}\oplus\mathcal{O}(1)),\quad 
\pr_{\pr^1\times\pr^1}(\mathcal{O}^{\oplus
  (n-2)}\oplus\mathcal{O}(1,1)), 
\quad\pr_{\mathbb{F}_1}(\mathcal{O}^{\oplus
  (n-2)}\oplus \mathcal{O}(l))),$$ 
where $l\subset \mathbb{F}_1$ is the proper transform of a general line in
$\pr^2$.

A different example is given by $\mathbb{F}_1\times\pr^1\times\pr^{n-3}$,
which has a quasi elementary contraction onto $\pr^1\times\pr^{1}$
with fiber $\pr^1\times\pr^{n-3}$. In this case the variety is a
product, but it is not the product of the fiber and the target of the
contraction. 
\end{example}
\begin{example}[(Elementary contractions over $3$-folds)]
Let $Y=\pr_{\pr^1\times\pr^1}(\mathcal{O}\oplus\mathcal{O}(1,1))$,
so that $Y$ is Fano with $\rho_Y=3$. Observe that $Y$ is the
divisorial resolution of a quadric $Q\subset\pr^4$ with an isolated
singularity. Let $L\in\Pic Y$ be the pull-back of
$\mathcal{O}_Q(1)$, and consider
$$X=\pr_{Y}(\mathcal{O}^{\oplus
  (n-3)}\oplus L).$$
Then $X$ is Fano with dimension $n$ and $\rho_X=4$, and it is not a product.
One can write down analogous examples with $\rho_Y=1,2$.

We do not know whether there are similar examples with
$\rho_Y=4,5$. Let us point out that smooth Fano $3$-folds $Y$ 
with $\rho_Y=5$ (respectively $\rho_Y=4$)
which are not products are given by just two families
(respectively~$12$), after
\cite{morimukai}. 
\end{example}
\begin{example}[(Conic bundles)]\label{conicbdl}
Let $X$ be a smooth Fano variety and $f\colon X\to Y$ an
equidimensional contraction with $\dim Y=n-1$.
Then $f$ is a conic bundle, and
it is quasi elementary if and only if it
is elementary.

If $f$ is a conic bundle, then $Y$ is smooth 
(see \cite[Theorem 3.1]{ando} and \cite[Proposition 4.1]{AW}) 
and if $\dim
X\leq 4$ then $Y$ is also Fano, by
 \cite[Corollary on p.\ 156]{wisn}. 
However this is not true in higher dimensions, see \cite[Example on
p.\ 156]{wisn}.  
\end{example}
\begin{example}[(Elementary contractions in dimension $4$)]
\label{none}
Let $X$ be a smooth Fano $4$-fold and consider an elementary
contraction
$f\colon X\to Y$  with $\dim Y=3$. 

If $f$ is equidimensional then $Y$ is smooth and Fano, see example
\ref{conicbdl}. 
However 
it is well known that $f$ can have isolated
$2$-dimensional fibers, which are  classified, see
\cite{AW,kachi}. 
In \cite[\S 11]{kachi} we  find several examples where $f$ is not
equidimensional and $Y\cong \pr^3$;
in particular in this case $X$ is not a product.
However we are not aware of similar examples with $Y$ singular.
\end{example}
\begin{example}[(A Fano $4$-fold with $\rho=6$ and only small
  elementary contractions)] \label{delpezzo} In the toric case,
smooth
Fano varieties are classified up to dimension $7$, the
cases of dimensions $5$, $6$ and $7$ 
 being quite recent  \cite{krnill,oebro}. They are a good
source of explicit examples.

After the classification in \cite{bat2} (see also \cite{sato}),
toric Fano 4-folds have Picard number at most $8$. The ones with
$\rho=7,8$ are just $S_3\times S_3$ and $S_3\times S_2$, $S_2$ and
$S_3$ being the blow-up of $\pr^2$ in two points and in three non
collinear points respectively.
Among the ones with $\rho=6$ there is a case 
with no (non trivial) quasi elementary contractions,
the toric Del Pezzo 4-fold $V$ (n.\ 118 in \cite{bat2}).

The Mori cone $\NE(V)$ has dimension $6$ but has $20$ extremal rays.
\emph{Every} elementary contraction is a small contraction with
exceptional locus a $\pr^2$ with normal bundle
$\ol_{\pr^2}(-1)^{\oplus 2}$. Every such exceptional $\pr^2$
intersects three others in a point. 

One can see that $V$ has a contraction of fiber type
$f\colon V\to Y$ with $\dim Y=3$ and $\rho_Y=1$, so $f$ is not quasi
elementary. There are $6$
two-dimensional fibers, which are unions of two exceptional $\pr^2$'s
intersecting in one point.  
Moreover $Y$
has $6$ isolated non $\Q$-factorial points in the images of these fibers. 

Up to our knowledge, among the known examples of Fano $4$-folds with no
(non trivial) quasi elementary contractions of fiber type, $V$ is the
one with largest Picard number.

This example has an analog $V_{n}$ in each even dimension $n=2m\geq
4$. This is a smooth toric Fano variety with $\rho_{V_n}=n+2$ and
$2\binom{n+1}{m}$ extremal rays. Every elementary contraction is a
small contraction with exceptional locus a $ \pr^m$ with normal bundle
$\ol_{\pr^m}(-1)^{\oplus m}$. The varieties $V_n$ are called toric Del Pezzo 
varieties and were introduced in \cite{VK}.
\end{example}
\small
\providecommand{\bysame}{\leavevmode\hbox to3em{\hrulefill}\thinspace}
\providecommand{\MR}{\relax\ifhmode\unskip\space\fi MR }
\providecommand{\MRhref}[2]{%
  \href{http://www.ams.org/mathscinet-getitem?mr=#1}{#2}
}
\providecommand{\href}[2]{#2}

\end{document}